%% file: 2d_isop_nonlocal_accepted_final.tex
\documentclass[a4paper,reqno]{amsart}

\usepackage{ifxetex,ifluatex}
\newif\ifxetexorluatex
\ifxetex
  \xetexorluatextrue
\else
  \ifluatex
    \xetexorluatextrue
  \else
    \xetexorluatexfalse
  \fi
\fi

\usepackage{lmodern}
\usepackage{xargs} 

\ifxetexorluatex
	\usepackage{polyglossia}
	\setmainlanguage{english}

	\usepackage{fontspec}
	\usepackage{amsmath} 
	\usepackage{amsthm}
	\usepackage{mathtools,thmtools} 
	\usepackage[mathrm=sym,mathbf=sym]{unicode-math}

	\ifluatex

	\DeclareMathAlphabet{\mathcal}{OMS}{cmsy}{m}{n}

	\newcommandx{\forcebold}[2][1=0.2]{{\pdfliteral direct {2 Tr #1 w}#2\pdfliteral direct {0 Tr 0 w}}}
	
	\else
	\DeclareMathAlphabet{\mathcal}{OMS}{cmsy}{m}{n}

	\newcommandx{\forcebold}[2][1=0.2]{{\boldmath #2}}
	\fi

	\renewcommand\boldsymbol\mathbf

\else
	\pdfoutput=1
	\usepackage[english]{babel}

	\usepackage[T1]{fontenc}
	\usepackage[utf8]{inputenc}
	\usepackage[final]{microtype}
	\usepackage{amsmath}
	\usepackage{amsthm}
	\usepackage{amssymb}
	\usepackage{mathtools,thmtools} 

	\usepackage{bm}
	\renewcommand{\mathbf}[1]{\bm{\mathrm{#1}}}
	
	\providecommand{\mathbfscr}[1]{\bm{\mathscr{#1}}}
	\providecommand{\mathbfcal}[1]{\bm{\mathcal{#1}}}
	\newcommandx{\forcebold}[2][1=0.2]{{\bm{#2}}}

	\usepackage{scalerel,stackengine}
	\stackMath
	\newcommand\widecheck[1]{%
	\savestack{\tmpbox}{\stretchto{%
	  \scaleto{%
		\scalerel*[\widthof{\ensuremath{#1}}]{\kern-.6pt\bigwedge\kern-.6pt}%
		{\rule[-\textheight/2]{1ex}{\textheight}}
	  }{\textheight}%
	}{0.5ex}}%
	\stackon[1pt]{#1}{\scalebox{-1}{\tmpbox}}%
	}

	\usepackage{mathrsfs,esint}

	\DeclareUnicodeCharacter{211D}{$\mathbb{R}$}
	\DeclareUnicodeCharacter{1D54A}{$\mathbb{S}$}
\fi

\usepackage[top=3cm, bottom=3cm, inner=2.7cm, outer=2.7cm,headsep=0.75cm,footskip=1.25cm]{geometry}

\usepackage{csquotes,enumitem,multicol}

\usepackage{hyperref}
\hypersetup{unicode,hypertexnames=false,colorlinks=true,linkcolor=blue,citecolor=blue,pdfauthor={B. Merlet, M. Pegon},pdftitle={Large mass rigidity for a liquid drop model in 2D with kernels of finite moments}}
\usepackage[capitalize,nameinlink,noabbrev]{cleveref}

\usepackage[hang,flushmargin]{footmisc}

\usepackage[final,color]{showkeys}
\renewcommand*\showkeyslabelformat[1]{
  \fbox{\parbox[t]{\marginparwidth}{\raggedright\normalfont\path{#1}}}}

\usepackage[backend=biber,style=numeric,isbn=false,alldates=year,url=false,doi=false,eprint=false,giveninits=true,maxbibnames=50]{biblatex}
\renewbibmacro{in:}{} 
\AtEveryBibitem{\clearfield{note}} 
\AtEveryBibitem{\clearlist{location}} 
\AtEveryBibitem{\clearfield{urlyear}} 
\AtEveryBibitem{\clearfield{pagetotal}} 

\renewbibmacro{volume+number+eid}{%
    \textbf{
	\printfield{volume}%
    \setunit{\adddot}%
    \printfield{number}%
    \setunit{\addcomma\space}%
    \printfield{eid}}
}
\AtBeginBibliography{%

}

\DeclareFieldFormat
  [article,inproceedings,incollection,report,unpublished,online]
  {title}{#1\addcomma}
\DeclareFieldFormat
  [book,inbook,incollection,unpublished]
  {date}{(#1)}
\DeclareFieldFormat
  [book,inbook]
  {title}{\textit{#1}\addcomma}
\DeclareFieldFormat
  [article,inbook,book,incollection,inproceedings,report,unpublished,online]
  {pages}{pp. #1}

\DeclareDelimFormat[bib]{nametitledelim}{\space:\space}
\DeclareNameAlias{labelname}{given-family} 

\usepackage{xifthen}
\usepackage[dvipsnames]{xcolor}
\makeatletter \def\SK@refcolor{\color{OrangeRed}} \makeatother
\makeatletter \def\SK@labelcolor{\color{Aquamarine}} \makeatother

\usepackage{graphicx}
\usepackage[mode=buildnew]{standalone}
\usepackage{tikz}
\usepackage{appendix}
\usepackage[colorinlistoftodos,textsize=small,disable=true]{todonotes}
\newcommandx{\unsure}[2][1=]{\todo[linecolor=red,backgroundcolor=red!25,bordercolor=red,#1]{#2}}
\newcommandx{\change}[2][1=]{\todo[linecolor=blue,backgroundcolor=blue!25,bordercolor=blue,#1]{#2}}
\newcommandx{\info}[2][1=]{\todo[linecolor=OliveGreen,backgroundcolor=OliveGreen!25,bordercolor=OliveGreen,#1]{#2}}
\newcommandx{\improvement}[2][1=]{\todo[linecolor=Plum,backgroundcolor=Plum!25,bordercolor=Plum,#1]{#2}}

\newcommand{\nhphantom}[1]{\sbox0{#1}\hspace{-\the\wd0}}

\newcommand{\numberset}[1]{{\mathbb{#1}}}
\newcommand{\N}{\numberset{N}}

\newcommand{\R}{\numberset{R}}

\newcommand{\dd}{\mathop{}\mathopen{}\mathrm{d}} 
\newcommand{\da}{\mathop{}\mathopen{}\mathrm{d}a} 
\newcommand{\db}{\mathop{}\mathopen{}\mathrm{d}b} 
\newcommand{\dx}{\mathop{}\mathopen{}\mathrm{d}x} 
\newcommand{\dy}{\mathop{}\mathopen{}\mathrm{d}y} 
\newcommand{\dz}{\mathop{}\mathopen{}\mathrm{d}z} 
\newcommand{\dt}{\mathop{}\mathopen{}\mathrm{d}t} 
\newcommand{\dr}{\mathop{}\mathopen{}\mathrm{d}r} 
\newcommand{\ds}{\mathop{}\mathopen{}\mathrm{d}s}

\newcommand{\dH}{\mathop{}\mathopen{}\mathrm{d}\mathscr{H}}

\newcommand{\bbS}{\mathbb{S}}

\newcommand{\bfK}{\mathbf{K}}

\newcommand{\calE}{{\mathcal{E}}}
\newcommand{\calF}{{\mathcal{F}}}

\newcommand{\calH}{{\mathscr{H}}}
\newcommand{\calS}{{\mathcal{S}}}

\newcommand{\calL}{{\mathscr{L}}}

\newcommand{\co}{{\mathrm{co}}}
\newcommand{\ind}{{\mathbf{1}}}
\newcommand{\Ren}{{\R^n}} 

\newcommand{\om}{\omega}
\newcommand{\eps}{\varepsilon}
\newcommand{\vphi}{\varphi}

\newcommand{\ol}{\overline}
\newcommand{\wt}{\widetilde}

\newcommand{\bv}{\mathrm{BV}}

\newcommand{\loc}{{\mathrm{loc}}}

\newcommand{\id}{{\mathrm{Id}}}

\newcommand{\compl}{{\mathrm{c}}}

\newcommand{\Lip}{{\mathrm{Lip}}}

\newcommand{\grad}{\nabla}
\newcommand{\subsq}{\subseteq}

\newcommand{\bmass}[2][]{\lbrack B\rbrack_{\ifthenelse{\isempty{#1}}{}{#1,}#2}}

\newcommand{\noblue}[1]{{#1}}

\newcommand{\edit}[1]{#1}

\DeclareMathOperator{\Per}{\mathit{P}}

\DeclareMathOperator{\spt}{\mathrm{spt}}

\DeclarePairedDelimiter{\abs}{\lvert}{\rvert}

\DeclarePairedDelimiter{\norm}{\lVert}{\rVert}
\DeclarePairedDelimiter{\seminorm}{\lbrack}{\rbrack}

\numberwithin{equation}{section} 

\declaretheorem[name=Theorem]{mainthm}
\declaretheorem[name=Theorem,numberwithin=section]{mainres}

\declaretheorem[name=Definition]{maindfn}
\declaretheorem[name=Theorem,numberwithin=section]{thm}
\declaretheorem[name=Lemma,numberlike=thm]{lem}
\declaretheorem[name=Proposition,numberlike=thm]{prp}
\declaretheorem[name=Corollary,numberlike=thm]{cor}
\declaretheorem[name=Definition,numberlike=thm,style=definition]{dfn}
\declaretheorem[name=Remark,numberlike=thm,style=remark]{rmk}

\crefname{equation}{}{}
\crefname{enumi}{}{}
\crefname{thm}{Theorem}{Theorems}
\crefname{mainthm}{Theorem}{Theorems}
\crefname{mainres}{Theorem}{Theorems}
\crefname{maindfn}{Definition}{Definitions}
\crefname{lem}{Lemma}{Lemmas}
\crefname{mainapp}{Application}{Applications}
\crefname{prp}{Proposition}{Propositions}
\crefname{cor}{Corollary}{Corollaries}
\crefname{dfn}{Definition}{Definitions}
\crefname{rmk}{Remark}{Remarks}
\crefname{conj}{Conjecture}{Conjectures}
\crefname{ex}{Example}{Examples}
\crefname{appsec}{Appendix}{Appendices}

\def\Xint#1{\mathchoice
{\XXint\displaystyle\textstyle{#1}}%
{\XXint\textstyle\scriptstyle{#1}}%
{\XXint\scriptstyle\scriptscriptstyle{#1}}%
{\XXint\scriptscriptstyle
\scriptscriptstyle{#1}}%
\!\int}
\def\XXint#1#2#3{{%
\setbox0=\hbox{$#1{#2#3}{\int}$}
\vcenter{\hbox{$#2#3$}}\kern-.5\wd0}}

\def\dashint{\Xint-}

\renewcommand{\leq}{\leqslant}
\renewcommand{\geq}{\geqslant}

\title[Large mass rigidity for a liquid drop model in 2D]{\textbf{Large mass rigidity for a liquid drop model in 2D with kernels of finite moments}}
\author{Benoit Merlet}
\author{Marc Pegon}
\address{Benoit Merlet, Univ. Lille, CNRS, Inria, UMR 8524 - Laboratoire Paul Painlevé, F-59000 Lille, France}
\email{benoit.merlet@univ-lille.fr}
\address{Marc Pegon, Univ. Lille, CNRS, Inria, UMR 8524 - Laboratoire Paul Painlevé, F-59000 Lille, France}
\email{marc.pegon@univ-lille.fr}
\date{}

\begin{document}

\begin{abstract}
Motivated by Gamow's liquid drop model in the large mass regime, we consider an isoperimetric
problem in which the standard perimeter $P(E)$ is replaced by $P(E)-\gamma P_\varepsilon(E)$, with
$0<\gamma<1$ and $P_\varepsilon$ a nonlocal energy such that $P_\varepsilon(E)\to P(E)$ as
$\varepsilon$ vanishes. We prove that unit area minimizers are disks for $\varepsilon>0$ small
enough.\\
More precisely, we first show that in dimension $2$, minimizers are necessarily convex, provided
that $\varepsilon$ is small enough. In turn, this implies that minimizers have nearly circular
boundaries, that is, their boundary is a small Lipschitz perturbation of the circle. Then, using a
Fuglede-type argument, we prove that (in arbitrary dimension $n\geq 2$) the unit ball in
$\mathbb{R}^n$ is the unique unit-volume minimizer of the problem among centered nearly spherical
sets. As a consequence, up to translations, the unit disk is the unique minimizer.\\
This isoperimetric problem is equivalent to a generalization of the liquid drop model for the atomic
nucleus introduced by Gamow, where the nonlocal repulsive potential is given by a radial,
sufficiently integrable kernel. In that formulation, our main result states that if the
first moment of the kernel is smaller than an explicit threshold, there exists a critical mass $m_0$
such that for any $m>m_0$, the disk is the unique minimizer of area $m$ up to translations.
This is in sharp contrast with the usual case of Riesz kernels, where the problem does not
admit minimizers above a critical mass.
\end{abstract}

\maketitle
\vspace{-0.5cm}

\tableofcontents

\vspace{-1cm}

\section{Introduction}
\addtocontents{toc}{\protect\setcounter{tocdepth}{1}}

Given a positive, radial, measurable kernel $G:\Ren\to [0,+\infty)$ with finite first moment (that
is, $\abs{x}G(x)\in L^1(\Ren)$), we consider the nonlocal perimeter functional $\Per_G$
(see e.g. \autocite{CN2018,BP2019}) defined on measurable sets $E\subsq\Ren$ by
\begin{equation}\label{eq:defPG}
\Per_G(E)\coloneqq 2\iint_{E \times \left(\Ren\setminus E\right)} G(x-y) \dx\dy
=\iint_{\Ren\times\Ren} \abs{\ind_E(x)-\ind_E(y)}G(x-y)\dx\dy.
\end{equation}
Here $\ind_E$ denotes the indicator function of $E$.\\
For $\eps>0$, we introduce the rescaled kernel $G_\eps(x)\coloneqq\eps^{-(n+1)}G(\eps^{-1}x)$,
$x\in\Ren$. As will be justified later, the first moment of $G$ is fixed to an explicit dimensional
constant (see \cref{Hint}) so that $\Per_{G_\eps}(E)$ converges to $P(E)$ as $\eps$ vanishes. Given
$\gamma\in (0,1)$ and $\eps>0$, we study the minimization problem
\begin{equation}\label{minrp}\tag{$\mathcal{P}_{\gamma,\eps}$}
\min \Bigg\{ P(E)-\gamma\Per_{G_\eps}(E)~:~\abs{E}=\abs{B_1}\Bigg\},
\end{equation}
over sets of finite perimeter $E$ in $\Ren$, where $\abs{E}$ denotes the volume of $E$ (which we
often call its mass), that is, its Lebesgue measure, and $B_1$ is the open unit ball of $\Ren$.

Let us emphasize the competition between the two terms. The perimeter is an attractive term
minimized by balls under volume constraint. On the contrary, if $G$ is radially decreasing, due to
the negative sign, the nonlocal term is maximized by balls\footnote{\edit{This can be seen by Riesz'
symmetric rearrangement, using e.g. \autocite[Chapter~3.7]{LL2001} and the fact that $G$ is equal to
its symmetric rearrangement in that case.}\\In fact, even if $G$ is not radially decreasing,
$G_\eps$ \enquote{concentrates} near the origin when $\eps$ is small, heuristically making
$(-\Per_{G_\eps})$ a repulsive term whenever $G$ is not identically equal to $0$.}, \edit{and there
exists no minimizer for the functional $(-\Per_{G_\eps})$ under volume constraint}. This competition
makes the minimization problem nontrivial, even when it comes to existence of minimizers.

Problem \cref{minrp} is closely linked with variations of Gamow's liquid drop model for the atomic
nucleus \edit{in the large mass regime. Indeed, thanks to the $\eps^{-(n+1)}$ factor in $G_\eps$,
changing variables in \cref{eq:defPG}, we have $\Per_{G_\eps}(E)=\eps^{n-1}\Per_G(\eps^{-1}E)$, so
that
\[
P(E)-\gamma \Per_{G_\eps}(E) = \eps^{n-1}\big( P(\eps^{-1}E)-\gamma \Per_G(\eps^{-1}E)\big),
\]
and \cref{minrp} is equivalent to the problem
\begin{equation}\label{minpb}\tag{$\mathcal{P}'_{\gamma,\eps}$}
\min \Bigg\{ P(F)-\gamma\Per_{G}(F)~:~\abs{F}=\frac{\abs{B_1}}{\eps}\Bigg\},
\end{equation}
in the sense that $E_\eps$ is a minimizer of \cref{minrp} if and only if $F_\eps\coloneqq
\eps^{-1}E_\eps$ is a minimizer of \cref{minpb}.
Then, if in addition we assume that~$G$ is integrable in $\Ren$, we may write
\begin{equation}\label{eq:equivgamow}
\frac{1}{2}\Per_{G}(F) = \norm{G}_{L^1(\Ren)}\abs{F}-\iint_{F\times F} G(x-y)\dx\dy,
\end{equation}
thus \cref{minpb} is in fact equivalent to
\begin{equation}\label{mingam}\tag{$\mathcal{G}_{m_\eps}$}
\min \left\{ P(F)+\iint_{F\times F} \wt{G}(x-y)~:~ \abs{F}=m_\eps\right\},
\end{equation}
where we have set $m_\eps\coloneqq \eps^{-1}\abs{B_1}$ and $\wt{G}\coloneqq 2\gamma G$}. When $n=3$
and $\wt{G}(x)=1/(8\pi\abs{x})$, this is Gamow's liquid drop model (see \autocite{CMT2017} for a
general overview); note however that in that case, the minimized functional cannot be rewritten as
the difference between the perimeter and a nonlocal perimeter, since $\wt{G}$ is not integrable at
infinity.
As a prototypical model for various physical systems involving the competition between short-range
attractive forces and long-range repulsive ones, generalizations of this model have gained
increasing interest during the past decade, in particular generalizations in higher dimensions,
where the Coulomb potential is replaced with Riesz potentials, that is,
$\wt{G}(x)=\abs{x}^{\alpha-n}$, $\alpha\in(0,n)$.
\noblue{In particular, it was shown that for every Riesz kernel, in the small mass regime, the
unique minimizer of the liquid drop model~\cref{mingam} is the ball, up to translations
(see~\autocite{KM2013,KM2014,Jul2014,BC2014,FFM+2015}).
Conversely, for $\alpha\in [n-2,n)$, the problem admits no minimizer above a critical mass (see
\autocite{BC2014,KM2013,KM2014,LO2014,FKN2016,FN2021}; see also \autocite{FL2015}).}
More general kernels of Riesz-type were studied e.g. in \autocite{CFP2020,NP2021,MW2021}, where the
unit ball is shown to be the unique minimizer in the small mass regime.

Although the small mass regime has been extensively studied, the literature on large mass
minimizers for Gamow-type problems is still sparse, since existence is rather unexpected in that
case, and is usually only recovered by adding an extra attractive potential, such as in
\autocite{ABCT2017,ABCT2018,GO2018}, or by adding a density to the perimeter, as in
\autocite{ABTZ2021}, where the authors show that if the density is a power-law growing sufficiently
fast at infinity, then minimizers always exist, and are balls for large masses.
It is worth mentioning that in the case of general kernels \textsl{with compact support}, the
author of \autocite{Rig2000} shows that minimizers exist \textsl{for all masses}.

Between Riesz kernels, which are not integrable at infinity, where \cref{mingam} does not admit
minimizers above a critical mass, and compactly supported kernels, where \cref{mingam} always admits
minimizers, it is natural to wonder what happens with non-compactly supported but
reasonably decaying kernels, such as Bessel kernels. These kernels behave as Riesz potentials near
the origin, but decrease exponentially at infinity. They were suggested in \autocite{KMN2016} as a
replacement for Riesz kernels for modeling diblock copolymers when long-range interactions are
partially screened by fluctuations in the background nuclear fluid density.
\edit{In \autocite{MW2021} (see Section 1.2 therein), motivated by some model for cell motility, the
authors suggest to study problem \cref{mingam} for rescalings of the kernel $G$, where $G$ is the
fundamental solution of $\left(\id +(-\Delta)^{s/2}\right)G=\delta_0$, for $s\in (0,2)$.
As pointed out in \autocite[Remark~1.3]{MW2021}, for $s\in(0,1)$ their asymptotic rescalings
correspond to the small mass regime, while for $s\in(1,2)$ they correspond to the large mass
regime. The authors focus on the case $s\in(0,1)$. Our work actually addresses the case
$s\in(1,2)$ (and with more general kernels).}

The study of the liquid drop model in the \textsl{large mass regime} for integrable kernels with
finite first moment (such as Bessel kernels or the ones just mentioned for $s\in(1,2)$), which is
equivalent to the study of \cref{minrp} when $\eps$ is small, has been started by the second author
in \autocite{Peg2021}. The existence of minimizers for any $\gamma\in(0,1)$ was established therein
for $\eps$ small enough, as well as the convergence of minimizers to the unit ball as $\eps$
vanishes. \noblue{It was conjectured there that the ball is actually the unique minimizer up to
translations, for $\eps$ small enough.  In this paper, we give a positive answer to this
conjecture in dimension $n=2$ under reasonable assumptions on the first moment of $G$ and the
second moment of $\grad G$, which are still satisfied by Bessel kernels. The conjecture remains
open in higher dimensions.}

Let us introduce the \enquote{critical energies}
\[
\calE_{G_\eps}:=P-\Per_{G_\eps}=\calF_{1,G_\eps},
\]
and define the energies
\begin{equation}\label{eq:defF}
\calF_{\gamma,G_\eps}(E)
\coloneqq (1-\gamma) P(E)+\gamma\calE_{G_\eps}(E)
= P(E)-\gamma\Per_{G_\eps}(E).
\end{equation}
Although the paper mostly deals with the \enquote{subcritical case} $\gamma<1$, we focus on the
critical energies in \cref{subsec:convexification}, and show that they decrease by convexification.
Finally, for any $k\in\N\setminus\{0\}$, we denote by
\begin{equation}\label{defmom}
I_G^{k} \coloneqq \int_{\Ren} \abs{x}^{k}\abs{\partial_r^{k-1} G(x)}\dx
\end{equation}
the $k$-th moment of the $(k-1)$-th radial derivative of the kernel $G$, whenever it is
well-defined.

In the paper, starting from \cref{sec:prelim}, \textbf{we shall always implicitly assume} that the
kernel $G$ satisfies the following general assumptions:
\begin{enumerate}[label=(H\arabic*)]
\item\label{Hrad} $G$ is nonnegative and radial, that is, there exists a measurable function
$g:(0,+\infty)\to[0,+\infty)$ such that $G(x)=g(\abs{x})$ for almost every $x\in\Ren$;
\item\label{Hint} the first moment of $G$ is finite and set to be
\[
I_G^1 = \frac{1}{\bfK_{1,n}},
\]
where, for any $\nu\in\bbS^{n-1}$, the constant $\bfK_{1,n}$ is defined by
\[
\bfK_{1,n}=\dashint_{\bbS^{n-1}} \abs{\sigma\cdot \nu}\dH^{n-1}_\sigma.
\]
\end{enumerate}
Starting from \cref{sec:fuglede}, we may explicitly use the extra assumption:
\begin{enumerate}[label=(H\arabic*)]
\setcounter{enumi}{2}
\item\label{Hhighermoments} $G\in W^{1,1}_\loc(\Ren\setminus\{0\})$, $I_G^2<\infty$, and
$g'(r)=O(r^{-(n+1)})$ at infinity.
\end{enumerate}

These assumptions are in particular satisfied by the Bessel kernels $\mathcal{B}_{\kappa,\alpha}$,
that is, fundamental solutions of the operators $(I-\kappa\Delta)^{\alpha/2}$, for
$\kappa,\alpha>0$ (see e.g.  \autocite[§~3.2]{Peg2021} or \autocite{Gra2014} for their definition
and properties).

\noblue{Even though our main result is in dimension $n=2$, where we prove that the unit disk is the
only minimizer up to translations, provided that $\eps$ is small enough, note that the intermediate
results of \cref{subsec:convnearlysph} and \cref{sec:fuglede} are obtained in arbitrary dimension.
That is, convex minimizers are nearly spherical sets, and the unit ball of~$\Ren$ is the unique
minimizer among nearly spherical sets whenever $\eps$ is small enough.}

Note also that the kernel $G$ is assumed to be radial but not necessarily radially nonincreasing, as
is often the case.
Let us emphasize that, contrarily to the small mass regime for Riesz-type potentials, here the
nonlocal perimeter term does not vanish in the limit but rather converges to a fraction of the
standard perimeter.

We shall now state the main result of the paper.

\begin{mainres}[Minimality of the unit disk]\label{mainthm:mindisk}
Assume $n=2$\edit{, $\gamma\in(0,1)$} and $G$ satisfies \crefrange{Hrad}{Hhighermoments}. Then there
exists $\eps_A=\eps_A(G,\gamma)>0$, such that, for every $\eps<\eps_A$, the unit disk is the unique
minimizer of \cref{minrp}, up to translations and Lebesgue-negligible sets.
\end{mainres}

\noblue{In terms of Gamow's problem \cref{mingam}, this means that if
\edit{$I_{\wt{G}}^1<2/\bfK_{1,2}=\pi$ (here $I_{\wt{G}}^1$ is the first moment of the kernel
$\wt{G}$ defined as in \cref{defmom})} and $\wt{G}$ is in addition integrable, \edit{then there
exists a critical mass $m_*$ such that the only solutions of \cref{mingam} with $m_\eps>m_*$ are the
disks of area $m_\eps$.}
In the particular case of the Bessel kernels $\wt{G}=\mathcal{B}_{\kappa,\alpha}$ with
$\kappa,\alpha>0$, such a critical mass exists whenever
\[
\kappa<\pi\left(\frac{\Gamma\left(\frac{\alpha}{2}\right)}
{\Gamma\left(\frac{1+\alpha}{2}\right)}\right)^2,
\]
(see \autocite[Corollary~3.9 \& Proposition~3.10]{Peg2021}).\medskip}

The proof of \cref{mainthm:mindisk} decomposes as follows. First we establish:

\begin{mainthm}[2D minimizers are convex]\label{mainthm:convexmin}
Assume $n=2$\edit{, $\gamma\in(0,1)$} and $G$ satisfies \cref{Hrad,Hint}. Then there exists
$\eps_1=\eps_1(G,\gamma)>0$ such that, for every $0<\eps<\eps_1$, \cref{minrp} admits a minimizer,
and every minimizer is convex, up to a Lebesgue-negligible set.
\end{mainthm}

The existence of minimizers for small $\eps$ was shown in \autocite{Peg2021}, where the second
author also proved that they are necessarily connected whenever $\eps$ is small enough.
The idea for proving the convexity of minimizers is to study the critical energy on the real line,
and show that it decreases by convexification and by expansion of segments, so that, by a slicing
argument, the critical energy of a \textsl{connected set in dimension~$2$} decreases after
convexification. As a consequence, since the perimeter of a connected set is also reduced by
convexification, so is $\calF_{\gamma,G_\eps}$. This slicing argument is specific to the dimension
$2$, where a line intersects a connected set if and only if it intersects its convex hull. This
fails in higher dimension.

Note that this is not enough to conclude that minimizers are convex. Indeed, although for every
minimizer $E_\eps\subset\R^2$ with $\eps$ small enough, we have
$\calF_{\gamma,G_\eps}(\co(E_\eps))\leq\calF_{\gamma,G_\eps}(E_\eps)$, where $\co(E_\eps)$ denotes
the convex hull of $E_\eps$, the volume of $\co(E_\eps)$ is larger than $\abs{B_1}$ if $E_\eps$ is
not convex.  However, using the fact that a minimizer $E_\eps$ is already close to the unit ball by
\autocite{Peg2021} and the convexity of $\co(E_\eps)$, we prove that, if $E_\eps$ is not convex,
scaling down $\co(E_\eps)$ to make its volume equal to $\abs{B_1}$ strictly
decreases the energy~$\calF_{\gamma,G_\eps}$, which contradicts the minimality of $E_\eps$.

The convexity of minimizers $E_\eps$ allows us to improve the convergence of $\partial E_\eps$
towards $\partial B_1$ as $\eps$ goes to $0$, from the previously known Hausdorff convergence to
Lipschitz convergence. We deduce that minimizers are nearly spherical sets, whose definition is
given just below.

\begin{maindfn}[Nearly spherical sets]\label{maindfn:nearlysph}
For $t\in (0,1/2)$, we say that $E\subsq\Ren$ is a centered $t$-nearly spherical set if
\[
\int_E x\dx=0,
\]
and if there exists $u\in \Lip(\bbS^{n-1})$ with $\norm{u}_{L^\infty(\bbS^{n-1})}+\norm{\grad_\tau\,
u}_{L^\infty(\bbS^{n-1})}\leq 1$ such that
\[
\partial E=\Big\{ \big(1+tu(x)\big)x~:~ x\in\bbS^{n-1}\Big\}.
\]
In dimension $n=2$, we use the terminology \enquote{$t$-nearly circular set} for \enquote{$t$-nearly
spherical set}.
\end{maindfn}

\begin{mainthm}[2D minimizers have nearly circular boundaries; see
{\cref{prp:minsph}}]\label{mainthm:minnearlycirc}
Assume $n=2$\edit{, $\gamma\in(0,1)$} and~$G$ satisfies \cref{Hrad,Hint}. There exist
$\eps_2=\eps_2(G,\gamma)>0$ and a function $t:(0,\eps_2)\to [0,1/2)$ depending only on $G$
and $\gamma$ such that
\begin{itemize}
\item $t(\eps)\to 0$ as $\eps\to 0$;
\item for every $\eps<\eps_2$, any minimizer $E$ of \cref{minrp} is, up to a translation and a
Lebesgue-negligible set, a centered $t(\eps)$-nearly circular set.
\end{itemize}
\end{mainthm}

\cref{mainthm:minnearlycirc} is a direct consequence of \cref{mainthm:convexmin} and the uniform
convergence of minimizers already shown in \autocite{Peg2021}, using the geometric fact that the
normal vectors to the boundary of a convex set lying between two balls~$B_r$ and $B_R$, $r< 1< R$
converge to those of the unit sphere as $r,R\to 1$ (see \autocite{Oss1987,Fug1989}).

We end the proof by showing that for $\eps$ and $t$ small enough, any centered $t$-spherical
minimizer of~\cref{minrp} is the unit ball. This last result is \textsl{not} specific to dimension
$n=2$.

\begin{mainthm}[Minimality of the unit ball among nearly spherical sets;
see {\cref{prp:fuglede2}}]\label{mainthm:minballsph}
Assume that \edit{$\gamma\in(0,1)$ and} $G$ satisfies \crefrange{Hrad}{Hhighermoments}. Then there exist
$t_*=t_*(n,G,\gamma)>0$ and $\eps_3=\eps_3(n,G,\gamma)>0$ such that, for every
$t<t_*$, if $E$ is a $t$-nearly spherical set, then we have
\[
\calF_{\gamma,\eps}(B_1)\leq \calF_{\gamma,\eps}(E),\qquad\forall 0<\eps<\eps_3,
\]
and the inequality is strict if $E\neq B_1$ (in the sense that they differ by a set positive
measure).
\end{mainthm}

\cref{mainthm:mindisk} is then an immediate consequence of
\crefrange{mainthm:convexmin}{mainthm:minballsph}.\medskip

\edit{The proof of \cref{mainthm:minballsph} relies on a bound (in our case, an upper bound) on the
quantity $\Per_{G_\eps}(E_t)-\Per_{G_\eps}(B_1)$ for a centered $t$-nearly spherical set $E_t$ with
$\partial E_t=\left\{(1+tu(x))~:~x\in\bbS^{n-1}\right\}$, in terms of the $L^2$ norms of $u$ and
$\grad_\tau\, u$ on the sphere. In the case of the local perimeter, this kind of control is
well-known and is originally due to \Citeauthor{Fug1989} (see \autocite[Theorem~1.2]{Fug1989}), who
proved
\[
\frac{t^2}{10}\left(\norm{u}_{L^2(\bbS^{n-1})}^2+\norm{\grad_\tau\, u}^2_{L^2(\bbS^{n-1})}\right)
\leq \frac{P(E_t)-P(B_1)}{P(B_1)}\leq
\frac{3}{5}\norm{\grad_\tau\,u}^2_{L^2(\bbS^{n-1})},
\]
provided that $t$ is small enough, depending only on $n$.
Similar results were obtained for so-called fractional perimeters $P_s$ as well as for Riesz
potentials in \autocite{FFM+2015} (see Theorems~2.1 and 8.1 therein), where the quantities are
bounded in terms of the $L^2$ norm and fractional Sobolev seminorms of $u$ on the sphere. Our
computations are inspired by the ones in \autocite{FFM+2015}, however, due to the general form of
the kernel $G$, they are more involved and quite tricky at times.}

\edit{As a last introductory remark, let us comment on the constants $t_*$, $\eps_A$, $\eps_i$,
$i\in\{1,2,3\}$ and their dependence in $\gamma$.
As expected, the constants vanish as $\gamma$ tends to $1$, and the rate at which they vanish
depends on the convergence rate of the quantity $P(B_1)-P_{G_\eps}(B_1)$ (quadratic in $\eps$) and
on the decay of the kernel $G$ at infinity. Assuming that $G(x)=O(\abs{x}^{-(n+1+\beta)})$ at
infinity, by \cref{rmk:depeps12} and \cref{rmk:depepsstar}, as $\gamma$ tends to $1$, we have}
\[
\edit{\begin{aligned}
\eps_i &\sim C(n,G)(1-\gamma)^{\max\left(\frac{3}{2},\frac{1}{2} +\frac{1}{\beta}\right)}\quad
\text{ for }i\in\{1,2\},\\
\eps_3&\sim C(n,G)(1-\gamma)^{\frac{1}{2}},\\
\eps_A&\sim C(n,G)(1-\gamma)^{\max\left(\frac{5}{2},\frac{1}{2},
+\frac{1}{\beta}\right)},\\
t_*& \sim C(n,G)(1-\gamma).
\end{aligned}}
\]

\subsection*{Outline of the paper}

The structure of the paper follows the strategy of the proof. In \cref{sec:prelim} we recall
some useful results from \autocite{Peg2021} on minimizers of \cref{minrp} and some facts on nonlocal
perimeters.
In~\cref{sec:minnearlycirc}, we show that 2D minimizers are convex and thus nearly circular
sets for small $\eps$, that is,~\cref{mainthm:convexmin} and \cref{mainthm:minnearlycirc}, where
the latter is a consequence of \cref{prp:minsph}. Finally, \cref{sec:fuglede} is dedicated to the
proof of \cref{mainthm:minballsph}, which is a consequence of \cref{prp:fuglede2}.

\subsection*{Notation}

\paragraph{\textsl{Operations on sets.}}
For any set $E\subsq\Ren$, $E^\compl\coloneqq \Ren\setminus E$ denotes its complement, $\co(E)$ its
convex hull (that is, the intersection of all convex sets containing $E$), and $\abs{E}$ its
Lebesgue measure, whenever $E$ is measurable.
We write $E\triangle F$ for the symmetric difference of $E$ and $F$, and $E\sqcup F$ for the union
of $E$ and $F$ whenever they are disjoint.

\vspace{10pt}

\paragraph{\textsl{Hausdorff measures.}}
We denote by $\calH^k$ the $k$-dimensional Hausdorff measure in $\Ren$. When integrating w.r.t. the
measure $\calH^k$ in a variable $x$, we use the notation $\dH_x^k$ instead of the more standard but
less compact $\dH^k(x)$.

\vspace{10pt}

\paragraph{\textsl{Balls and spheres.}}
We denote by $B_r(x)$ the open ball in $\Ren$ of radius $r$ centered at $x$. For brevity, we write
$B_r$ when $x$ is the origin.
The volume of $B_1$ is $\om_n\coloneqq\abs{B_1} =
\frac{\pi^{n/2}}{\Gamma\left(1+\frac{n}{2}\right)}$, and the area of the unit sphere
$\bbS^{n-1}$ is $\calH^{n-1}(\bbS^{n-1})=n\om_n$, which we also write $\abs{\bbS^{n-1}}$ for
simplicity.

\vspace{10pt}

\paragraph{\textsl{Sets of finite perimeter.}}

We denote by $\bv(\Ren)$ the space of functions with bounded variation in $\Ren$. For any $f\in
\bv(\Ren)$ we let $\abs{Df}$ be its total variation measure, and set
$\seminorm{f}_{\bv(\Ren)}\coloneqq\int_{\Ren}\, \abs{Df}$.
For a set of finite perimeter $E$ in $\Ren$, we let $\ind_E\in\bv(\Ren)$ be its characteristic
function (i.e., $\ind_E(x)=1$ if $x\in E$ and $0$ otherwise), and define its perimeter by
$P^n(E)\coloneqq \int_{\Ren}\,\abs{D\,\ind_E}$. When there can be no confusion, we may drop the
superscript and simply write $P(E)$ for the perimeter functional in $\Ren$.
We denote by~$\mu_E\coloneqq D\,\ind_E$ the Gauss--Green measure associated with the set of finite
perimeter $E$, and by~$\nu_E(x)$ the outer unit normal of $\partial^* E$ at $x$, where $\partial^*E$
stands for the reduced boundary of~$E$.
We refer to e.g.~\autocite[Chapter~5]{EG2015} or \autocite{Mag2012} for further details on functions
of bounded variations and sets of finite perimeter.

\vspace{10pt}

\addtocontents{toc}{\protect\setcounter{tocdepth}{2}}

\section{Preliminaries}\label{sec:prelim}

From now on we \edit{shall always implicitly assume that $\gamma\in(0,1)$ and that $G$ satisfies
\cref{Hrad,Hint}}.

For a general nonnegative radial kernel $K$ with finite first moment, we have the following control
of~$\Per_K$ by the perimeter, as an immediate consequence of \autocite[Proposition~3.1]{Peg2021} and
of the second expression of the nonlocal perimeter given by \cref{eq:defPG}.

\begin{prp}\label{prp:boundPerG}
\edit{Let $K:\Ren\to [0,+\infty)$ be a kernel satisfying the same assumptions \cref{Hrad,Hint} as $G$,
except that the value of its first moment is not prescribed.}
Then, for every set of finite perimeter $E$ in $\R^n$, we have
\[
\Per_{K}(E)\leq \bfK_{1,n}I_K^1 P(E).
\]
\edit{In particular, for the kernels $G_\eps$, we have}
\[
\Per_{G_\eps}(E)\leq P(E),\qquad\forall \eps>0.
\]
\end{prp}

We also have the following convergence result, which is a consequence of \autocite{Dav2002} and our
choice of $I_G^1$.

\begin{prp}\label{prp:limPerG}
For any set of finite perimeter $E$ in $\Ren$, we have
\begin{equation}\label{limPerG:res}
\Per_{G_\eps}(E)~\xrightarrow{\eps\to 0}~ P(E).
\end{equation}
\end{prp}

We will use the following computation obtained in \autocite[Lemma~3.5]{Peg2021}, which clarifies
the behavior of the nonlocal perimeter under scaling.

\begin{lem}\label{lem:derivpg}
For any set of finite perimeter $E\subsq\Ren$, the function $t\mapsto \Per_{G_\eps}(tE)$ is locally
Lipschitz continuous in $(0,+\infty)$, and for almost every $t$, we have
\[
\begin{aligned}
\frac{\dd}{\dt} \left[\Per_{G_\eps}(tE)\right]
&= \frac{n}{t}\Per_{G_\eps}(tE)-\frac{1}{t}\wt{P}_{G_\eps}(tE),
\end{aligned}
\]
where $\wt{P}_{G_\eps}(E)$ is defined by 
\begin{equation}\label{defReps}
\wt{P}_{G_\eps}(E):= 2\int_E \int_{\partial^* E} G_\eps(x-y)\,(y-x)\cdot\nu_E(y)\dH^{n-1}_y\dx.
\end{equation}
\end{lem}

Let us remark that in \autocite{Peg2021}, $G$ is assumed to be in addition integrable in $\Ren$, but
\cref{lem:derivpg} can be deduced by approximating $G$ with maps $G_k:x\mapsto
\chi_k(\abs{x})G(x) \in L^1(\Ren)$. Indeed, \mbox{let $\chi_k\in C^\infty(\R^+,[0,1])$} be cutoff
functions with $\chi_k(r)=0$ for $r\leq 1/k$, $\chi_k(r)=1$ for $r\geq \frac{2}{k}$, so
that~$I_{G_k}^1\leq I_G^1$, $x\mapsto \abs{x}G_k(x)$ converges to $x\mapsto \abs{x}G(x)$ in
$L^1(\Ren)$, and notice that
\[
\abs{\Per_G(E)-\Per_{G_k}(E)} \leq P(E)\int_{\Ren} \abs{x}\abs{(G-G_k)(x)}\dx,
\]
and
\[
\abs{\wt{\Per}_G(E)-\wt{\Per}_{G_k}(E)} \leq 2P(E)\int_{\Ren} \abs{x}\abs{(G-G_k)(x)}\dx.
\]
In order to study the minimality of the unit ball among nearly spherical sets, we will use the
following Bourgain-Brezis-Mironescu-type result (see \autocite{BBM2001}) for approximating the $H^1$
seminorm on the sphere by nonlocal seminorms.

\begin{lem}\label{lem:cvsphunif}
Let us define the $(n-1)$-dimensional approximation of identity $(\eta_\eps)_{\eps>0}$ by
\[
\eta(t)\coloneqq 2t^2g(t),\qquad\text{ and }\qquad
\eta_\eps(t)\coloneqq\eps^{-(n-1)}\eta(\eps^{-1}t), \qquad\forall t>0,~\forall \eps>0.
\]
When $n=2$, we assume in addition that $g$ is such that the family $(\eta_\eps)_{\eps>0}$ satisfies
\begin{equation}\label{eq:Tcond}
\sup_{r\in (R,2)} \eta_\eps(r)\,\xrightarrow{\eps\to 0}\, 0,\qquad\forall R\in(0,2).
\end{equation}
Then for any $u\in H^1(\bbS^{n-1})$, we have
\[
\begin{aligned}
\iint_{\bbS^{n-1}\times\bbS^{n-1}} \frac{(u(x)-u(y))^2}{\abs{x-y}^2} \eta_\eps(\abs{x-y})
\dH_x^{n-1}\dH_y^{n-1}
\leq \big(1+q_\eta(\eps)\big) \int_{\bbS^{n-1}} \abs{\grad_\tau\, u}^2\dH^{n-1},
\end{aligned}
\]
where $q_\eta(\eps)$ vanishes as $\eps$ goes to $0$, and depends only on $n$ and $G$. In addition,
for any $u\in H^1(\bbS^{n-1})$,
\[
\begin{aligned}
\iint_{\bbS^{n-1}\times\bbS^{n-1}} \frac{(u(x)-u(y))^2}{\abs{x-y}^2} \eta_\eps(\abs{x-y})
\dH_x^{n-1}\dH_y^{n-1}
\,\xrightarrow{\eps\to 0}\, \int_{\bbS^{n-1}} \abs{\grad_\tau\, u}^2\dH^{n-1}.
\end{aligned}
\]
\end{lem}

\begin{proof}
One easily checks that assumptions \cref{Hrad,Hint} ensure that the family $(\eta_\eps)_{\eps>0}$ is
a $(n-1)$-dimensional approximation of identity, up to multiplication by the constant
$\bfK_{2,n-1}=1/(n-1)$, i.e.,
\begin{enumerate}[label=(\roman*),parsep=0pt,itemsep=0.5pt]
\begin{multicols}{2}
\item $\displaystyle \abs{\bbS^{n-2}}\int_0^{+\infty} \eta_{\eps}(r)r^{n-2}\dr =
\frac{1}{\bfK_{2,n-1}}$; \item$\displaystyle\lim_{\eps \to 0}\, \int_\delta^{+\infty}
\eta_{\eps}(r)r^{n-2} \dr= 0,\quad \forall
\delta > 0$.
\end{multicols}
\end{enumerate}
These properties (together with \cref{eq:Tcond} when $n=2$) allow us to apply
\autocite[Propositions~A.1~\&~A.4]{Peg2021}, which gives the result.
\end{proof}

\begin{rmk}\label{rmk:condTHinfty}
If \cref{Hhighermoments} stands true, then the condition \cref{eq:Tcond} is satisfied, in particular
when $n=2$. Indeed, with \cref{Hint}, \cref{Hhighermoments} implies that $g\in
W^{1,1}_\loc(0,+\infty)$ and the functions $t\mapsto t^ng(t)$ and $t\mapsto t^{n+1}g'(t)$ are
integrable on $(0,+\infty)$. In addition, integrating the function $(t^{n+1}g(t))'$ between $r$ and
$R$, we have the relation
\[
R^{n+1}g(R)-r^{n+1}g(r)=(n+1)\int_r^R t^ng(t)\dt+\int_r^R t^{n+1}g'(t)\dt.
\]
Since $t^ng(t)$ and $t^{n+1}g'(t)$ are integrable on $(0,+\infty)$, this implies that $r^{n+1}g(r)$
has a limit in $0^+$ and at infinity.
By the integrability of $t^n g(t)$ on $(0,+\infty)$, these limits are necessarily $0$. In
particular, $r^{n+1}g(r)\xrightarrow{r\to\infty} 0$, so that
\[
\eta_\eps(r)=r^2\eps^{-(n+1)}g(\eps^{-1}r)=r^{1-n}(\eps^{-1}r)^{-(n+1)}g(\eps^{-1}r)
\]
vanishes uniformly on $(R,2)$ as $\eps\to 0$, for every $R\in(0,2)$.
\end{rmk}

Eventually, gathering results from \autocite[Theorems~A and B]{Peg2021} (see also Theorem~4.16
therein), we have existence and convergence results for minimizers of \cref{minrp}.
We also know that minimizers are connected for small $\eps$. Here connectedness is to be understood
in a measure-theoretic sense for sets of finite perimeter, often referred to as indecomposability,
as defined below (see \autocite{ACMM2001}).

\begin{dfn}
We say that a set of finite perimeter $E$ is decomposable if there exist two sets of finite
perimeter $E_1$ and $E_2$ such that $E=E_1\sqcup E_2$, $\abs{E_1}>0$, $\abs{E_2}>0$ and
$P(E)=P(E_1)+P(E_2)$. Naturally, we say that a set of finite perimeter is indecomposable if it is
not decomposable.
\end{dfn}

Let us remark that by \autocite[Theorem~2]{ACMM2001}, the notion of connectedness and
indecomposability coincide whenever $E$ is an \textsl{open} set of finite perimeter such that
$\calH^{n-1}(\partial E)=\calH^{n-1}(\partial^* E)$.

\begin{thm}\label{thm:existmin}
There exist $\eps_0=\eps_0(n,G,\gamma)>0$ and a function
$\delta=\delta(n,G,\gamma):(0,+\infty)\to (0,1/4)$ vanishing in $0^+$ such that the
following holds. For every $0<\eps<\eps_0$, \cref{minrp} admits a minimizer. In addition, any such
minimizer $E_\eps$ is indecomposable, and up to a translation and a Lebesgue-negligible set, it
satisfies
\begin{equation}\label{existmin:doubleincl}
B_{1-\delta(\eps)} \subsq E_\eps \subsq B_{1+\delta(\eps)}.
\end{equation}
In dimension $n=2$, any minimizer $E_\eps$ with $0<\eps<\eps_0$ is Lebesgue-equivalent to a
connected set which still satisfies \cref{existmin:doubleincl}.
\end{thm}

\begin{proof}
In \autocite{Peg2021}, the kernel $G$ is assumed to be integrable in $\Ren$. However, it
is actually only required for the two following reasons: first, to be able to write
\cref{eq:equivgamow} and obtain the equivalence with the Gamow-type minimization problem
\cref{mingam}; second, by this equivalence, to deduce that minimizers of~\cref{minrp} are so-called
quasi-minimizers of the perimeter, and thus are (non-uniformly in $\eps$)
$C^{1,1/2}$-regular outside a \enquote{small} singular set. Here, we do not need the
equivalence with \cref{mingam} nor the \textsl{a priori} regularity of minimizers.
In the end, apart from the $C^{1,1/2}$ partial regularity of minimizers, all the conclusions
of \autocite[Theorems~A and B]{Peg2021} follow. More precisely, there exists
$\eps_0=\eps_0(n,G,\gamma)$ such that, for any~$0<\eps<\eps_0$, \cref{minrp} admits a minimizer.
In addition, any such minimizer $E_\eps$ is indecomposable and, up to a translation and
Lebesgue-negligible set, it satisfies
\[
B_{1-\delta(\eps)} \subsq E_\eps \subsq B_{1+\delta(\eps)},
\]
where $\delta:(0,+\infty)\to(0,1/4)$ is a function depending only on $n$, $G$, and $\gamma$
vanishing in $0^+$.

To conclude, there remains to show that in dimension $n=2$, $E_\eps$ is equivalent to a connected
set with~\cref{existmin:doubleincl}, that is, to link the indecomposability of $E_\eps$ with the
topological notion of connectedness. It is not a trivial question, at least without (weak)
regularity results on minimizers.
However, \autocite[Theorem~8]{ACMM2001} shows that in dimension $2$, $\wt{E}_\eps\coloneqq
\mathring{E}^M_\eps\setminus\partial^S E_\eps$ is connected, where $\mathring{E}^M_\eps$ is the
measure-theoretic interior of $E_\eps$, and
\[
\partial^S E_\eps\coloneqq \left\{ x\in\R^2~:~ \limsup_{r\to 0^+}~ \frac{\calH^{1}(\partial^* E_\eps
\cap B_r(x))}{r}>0\right\}.
\]
Since $\calH^{1}(\partial^S E_\eps\setminus\partial^* E_\eps)=0$ and
$\calL^2(E_\eps\triangle \mathring{E}_\eps^M)=0$, we have $\calL^2(E_\eps\triangle\wt{E}_\eps)=0$,
and since $B_{1-\delta(\eps)}\subsq\mathring{E}^M_\eps\subsq B_{1+\delta(\eps)}$ and $\partial^S
E_\eps \subsq \ol{B}_{1+\delta(\eps)}\setminus B_{1-\delta(\eps)}$, $\wt{E}_\eps$ satisfies
\cref{existmin:doubleincl}.
\end{proof}

\begin{rmk}
In higher dimensions $n\geq 3$, one could show as well that any minimizer of \cref{minrp} is
equivalent to a connected set for small $\eps$, without assuming that $G\in L^1(\Ren)$. Indeed,
proceeding e.g. as in \autocite[Lemma~5.6]{MS2019}, it is possible to obtain \textsl{uniform}
(w.r.t. $\eps$) density estimates for minimizers, and with those to deduce that any minimizer is
equivalent to an open set $E_\eps$ such that $\partial E_\eps=\spt \mu_\eps$. The indecomposability
of $E_\eps$ then implies connectedness by \autocite{ACMM2001}.
\end{rmk}

\begin{rmk}\label{rmk:distball}
In view of \autocite[Lemma~4.1]{Peg2021} and the proofs of Section~4.1 therein, we see that
\[
\delta(\eps)^2 = C(n)\left(\frac{\gamma}{1-\gamma}\right)\big(P(B_1)-\Per_{G_\eps}(B_1)\big),
\]
and that $\eps_0$ is actually chosen so that $\delta(\eps)\leq C(n)$ for every $0<\eps<\eps_0$, for
some dimensional constant~$C(n)$.
A computation shows that $P(B_1)-\Per_{G_\eps}(B_1)\sim C(n,G)\eps^2$ as $\eps$ vanishes, thus
\[
\delta(\eps)^2 \sim C(n,G)\left(\frac{\gamma}{1-\gamma}\right)\eps^2,
\]
and we can choose 
\[
\eps_0=C(n,G)\left(\frac{1-\gamma}{\gamma}\right)^{\frac{1}{2}}.
\]
Hence, the closer $\gamma$ is to $1$, the smaller $\eps_0$ is.
\end{rmk}


%
%

\section{Minimizers are nearly circular sets in dimension \texorpdfstring{$2$}{2}}
\label{sec:minnearlycirc}


\subsection{Decrease of the critical energy by convexification}\label{subsec:convexification}

We can recover the nonlocal perimeter of a measurable set $E\subsq\Ren$ by integrating the
$1$-dimensional nonlocal perimeter of all $1$-dimensional slices of $E$ in a given direction, and
averaging over all the directions.

\begin{prp}
For any measurable set $E\subsq\edit{\R^{n}}$, we have
\[
\begin{aligned}
\Per_{G_\eps}(E)
&=\int_{\bbS^{n-1}} \int_{\{\sigma\}^\perp} \left(\iint_{E_{\sigma,y}\times
\left(\R\setminus E_{\sigma,y}\right)} \abs{s-t}^{\edit{n-1}}g_\eps(\abs{s-t})\ds\dt\right)
\dH^{n-1}_y\dH^{n-1}_\sigma\\
&=\frac{1}{2\om_{n-1}}\int_{\bbS^{n-1}}\left(\int_{\{\sigma\}^\perp}
\Per^{1}_{\rho_\eps}(E_{\sigma,y}) \dH^{n-1}_y\right)\dH^{n-1}_\sigma,
\end{aligned}
\]
where \edit{$g_\eps\coloneqq\eps^{-(n+1)}g(\eps^{-1}\,\cdot)$},
\begin{equation}\label{eq:defEst}
E_{\sigma,y} := \Big\{ s\in\R~:~ y+s\sigma\in E\Big\},
\end{equation}
$\om_{n-1}$ is the volume of the unit ball in $\R^{n-1}$, and $\Per^{1}_{\rho_\eps}$ is the
$1$-dimensional nonlocal perimeter in $\R$ associated with the kernel $\rho_\eps$ defined by
$\rho_\eps(r)\coloneqq \om_{n-1}\abs{r}^{n-1}g_\eps(\abs{r})$ for $r\in\R\setminus\{0\}$, that is,
\[
\Per^1_{\rho_\eps}(J) \coloneqq \iint_{\R\times\R} \abs{\ind_J(s)-\ind_J(t)} \rho_\eps(s-t)\ds\dt,
\]
for every measurable set $J\subsq\R$.
\end{prp}

\begin{proof}
By the change of variable $y=x+r\sigma$ with fixed $x$ and Fubini's theorem, we have
\[
\begin{aligned}
\Per_{G_\eps}(E)
&= 2\iint_{E\times E^\compl} \edit{G_\eps}(x-y)\dx\dy\\
&= 2\int_{\bbS^{n-1}} \int_0^{+\infty} \int_E \ind_{E^\compl}(x+r\sigma)\, r^{n-1}g_\eps(r)
\dx\dr\dH^{\edit{n-1}}_\sigma\\
&=\frac{1}{\om_{n-1}}\int_{\bbS^{n-1}} \int_\R \int_E \ind_{E^\compl}(x+r\sigma)\,
\rho_\eps(r)\dx\dr\dH^{\edit{n-1}}_\sigma,
\end{aligned}
\]
where we have used the definition of $\rho_\eps$ for the last equality.
Then, for $\sigma$ fixed, let us make the change of variable $x=y+s\sigma$, where
$y=\pi_{\{\sigma\}^\perp}(x)$ is the orthogonal projection of $x$ on $\{\sigma\}^\perp$.
This yields
\[
\Per_{G_\eps}(E) = \frac{1}{2}\int_{\bbS^{n-1}}\int_{\{\sigma\}^\perp} \int_\R \int_\R
\ind_E\big(s\sigma+y\big)\ind_{E^\compl}\big((s+r)\sigma+y\big) \rho_\eps(r)
\ds\dr\dH^{n-1}_y\dH^{n-1}_\sigma.
\]
Finally, using Fubini's theorem and making the change of variable $t=s+r$, where $s$ is fixed, we
obtain, by definition of $E_{\sigma,y}$ and $\Per_{\rho_\eps}^1$,
\begin{multline*}
\Per_{G_\eps}(E) =
\frac{1}{\om_{n-1}}\int_{\bbS^{n-1}}\int_{\{\sigma\}^\perp}\int_{E_{\sigma,y}}
\int_{\R\setminus E_{\sigma,y}} \rho_\eps(s-t) \ds\dt\dH^{n-1}_y\dH^{n-1}_\sigma\\
=\frac{1}{2\om_{n-1}}\int_{\bbS^{n-1}}\int_{\{\sigma\}^\perp}\Per^1_{\rho_\eps}(E_{\sigma,y})
\dH^{n-1}_y \dH^{n-1}_\sigma.
\end{multline*}
This concludes the proof.
\end{proof}

\begin{rmk}\label{rmk:PerGbound}
Recall that by \autocite[Lemma~3.13]{Peg2021},
$\bfK_{1,n}=\frac{\Gamma\left(\frac{n}{2}\right)}{\sqrt{\pi}\Gamma\left(\frac{n+1}{2}\right)}$,
so that, by \cref{Hint}, the kernel $\rho_\eps$ satisfies
\[
\int_{\R} \abs{t}\rho_\eps(t)\dt
=2\om_{n-1}\int_0^\infty t^ng_\eps(t)\dt
=\frac{2\om_{n-1}I_G^1}{\abs{\bbS^{n-1}}}=\frac{2\om_{n-1}}{\bfK_{1,n}\abs{\bbS^{n-1}}}
=1
=\frac{1}{\bfK_{1,1}}.
\]
Hence by \cref{prp:boundPerG}, we have
\[
\calE^1_{\rho_\eps}(E)= P(E)-\Per^1_{\rho_\eps}(E)\geq 0
\]
for every measurable set $E\subsq\R$.
\end{rmk}

\edit{Similarly, as a straightforward consequence of \autocite[Theorem~3.103]{AFP2000} (see also
Crofton's formula \autocite[§3.16]{Mor2016} or \autocite[Theorem~3.2.26]{Fed1996}),}
for any set of finite perimeter $E\subsq\edit{\R^{n}}$, $E_{\sigma,y}$ is a one-dimensional set of
finite perimeter for $\calH^{n-1}$-almost every $\sigma$ and $y$, and we have
\[
P(E)=\calH^{n-1}(\partial^* E)=\frac{1}{2\om_{n-1}}\int_{\bbS^{n-1}} \int_{\{\sigma\}^\perp}
P^1(E_{\sigma,y})\dH_y^{n-1}\dH^{n-1}_\sigma,
\]
where $P^1(E_{\sigma,y})=\calH^0(\partial^* E_{\sigma,y})$ is the standard perimeter in dimension
$1$. Hence, we have the following representation of the critical energy $\calE_{G_\eps}$.

\begin{cor}\label{cor:expecrit}
For any set of finite perimeter $E\subsq\Ren$, we have
\[
\begin{aligned}
\calE_{G_\eps}(E)
&=\frac{1}{2\om_{n-1}} \int_{\bbS^{n-1}} \int_{\{\sigma\}^\perp}
\calE^1_{\rho_\eps}(E_{\sigma,y}) \dH_y^{n-1}\dH^{n-1}_\sigma,
\end{aligned}
\]
where $\calE^1_{\rho_\eps}:=P^1-\Per^1_{\rho_\eps}$, and $E_{\sigma,y}$ is given by
\cref{eq:defEst}.
\end{cor}

We give a simple expression of the one-dimensional critical energy.  of a segment $(a,b)\subsq\R$.

\begin{lem}\label{lem:critenergy1}
For every $a,b\in \R$ such that $a<b$, we have
\begin{equation}\label{eq:critenergy1}
\calE^1_{\rho_\eps}((a,b))=4\int_{-\infty}^a\int_b^{+\infty} \rho_\eps(s-t)\ds\dt.
\end{equation}
In particular $\calE^1_{\rho_\eps}((a,b))$ decreases as the interval grows.
In addition,
\begin{equation}\label{eq:critenergy2}
\calE^1_{\rho_\eps}(\emptyset)=\calE^1_{\rho_\eps}(\R)=\calE^1_{\rho_\eps}((b,+\infty))
=\calE^1_{\rho_\eps}((-\infty,a))=0.
\end{equation}
\end{lem}

\begin{proof}
By a change of variable and Fubini's theorem, for any $a\in\R$, we have
\[
P_{\rho_\eps}^1((-\infty,a))=2\int_{-\infty}^a\int_a^{+\infty} \rho_\eps(t-s)\ds\dt
=2\int_0^{+\infty} \rho(t) \left(\int_0^t \ds\right)\dt=2\int_0^{+\infty} t\rho(t)\dt=1.
\]
Similarly, $P_{\rho_\eps}^1((b,+\infty))=1$, and since $P^1((-\infty,a))=P^1((b,+\infty))=1$,
\cref{eq:critenergy2} follows.
Next,
\begin{equation}\label{eq:p1form1}
\begin{aligned}
2=P^1((a,b))
&=P^1((-\infty,a))+P^1((b,+\infty))\\
&=2\int_{-\infty}^a\int_a^{+\infty} \rho_\eps(s-t)\ds\dt+2\int_{-\infty}^b\int_b^{+\infty}
\rho_\eps(s-t)\ds\dt,
\end{aligned}
\end{equation}
for every $a,b\in\R$ such that $a<b$. We also have
\begin{equation}\label{eq:p1form2}
P^1_{\rho_\eps}((a,b))=2\int_{-\infty}^a \int_a^b \rho_\eps(s-t)\ds\dt+2\int_b^{+\infty} \int_a^b
\rho_\eps(s-t)\ds\dt.
\end{equation}
Subtracting \cref{eq:p1form2} from \cref{eq:p1form1}, we obtain \cref{eq:critenergy1}. The
identities \cref{eq:critenergy2} are obvious.
\end{proof}

For a general set of finite perimeter in $\R$, we have the following expression of the critical
energy.

\begin{lem}\label{lem:crit1d}
For any set $E\subset\R$ which is a finite disjoint union of open intervals, let
$\{C_i\}_{i\in\{1,\dotsc,N\}}$ be the connected components of $E^\compl$. Then we have
\begin{equation}\label{crit1d:res}
\calE^1_{\rho_\eps}(E)=2\sum_{\substack{1\leq i,j\leq N\\i\neq j}} \iint_{C_i\times C_j}
\rho_\eps(s-t)\ds\dt + \sum_{i=1}^N \calE_{\rho_\eps}^1(C_i),
\end{equation}
where, for the intervals $C_i$, $\calE_{\rho_\eps}^1(C_i)$ is given by \cref{lem:critenergy1}.
\end{lem}

Let us point out that $E^\compl$ may have up to two \textsl{unbounded} connected
components, and their critical energy is $0$ by \cref{lem:critenergy1}. As a consequence of
\cref{lem:critenergy1,lem:crit1d}, the 1D critical energy decreases by convexification and the
energy of a nonempty open segment is a decreasing function of its length.

\begin{cor}\label{cor:critdecconv}
Let $E\subsq J\subsq\R$ with $J$ an interval and $\calL^1(E)>0$, then
$\calE^1_{\rho_\eps}(J)\leq\calE^1_{\rho_\eps}(E)$.
\end{cor}

\begin{proof}
Let $E\subsq\R$ with $\calL^1(E)>0$, and let $J$ be an interval containing $E$.
Let $a'\leq a<b\leq b'\in \R\,\cup\,\{\pm\infty\}$ such that $\co(E)=(a,b)$ and $J=(a',b')$.
If $a'=-\infty$ or $b'=+\infty$, $\calE_{\rho_\eps}^1(J)=0\leq \calE^1_{\rho_\eps}(E)$ by
\cref{lem:critenergy1,rmk:PerGbound}, so the result holds true.
If $E$ does not have finite perimeter, then~${\calE^1_{\rho_\eps}(E)=+\infty}$, and the result holds
true as well. Thus, let us assume that $E$ is a bounded set of finite perimeter.
In particular, up to a negligible set, $E$ is the disjoint union of $k$ open intervals with $k\geq
1$ since $\calL^1(E)>0$. In addition, since $E$ is bounded, $E^\compl$ has two
unbounded components~${C_1=(-\infty,a)}$ and $C_{k+1}=(b,+\infty)$ (up to renumbering).
By \cref{lem:crit1d,lem:critenergy1}, we have
\[
\calE^1_{\rho_\eps}(E)\geq 4\int_{C_1\times C_{k+1}} \rho_\eps(s-t)\ds\dt
=\calE^1_{\rho_\eps}((a,b))
\geq \calE^1_{\rho_\eps}(J).
\]
\end{proof}

\begin{proof}[Proof of {\cref{lem:crit1d}}]
Let
\[
E=\bigsqcup_{i=1}^k~ (a_i,b_i),
\]
with $-\infty\leq a_1<b_1<\ldots<a_k<b_k\leq +\infty$, so that, setting $b_0\coloneqq-\infty$ and
$a_{k+1}\coloneqq+\infty$, the connected components of $E^\compl$ are given by
\[
C_i=(b_i,a_{i+1}),\qquad\forall i\in \{0,\dotsc,k\}.
\]
Beware that $C_0$ and $C_k$ are either empty or unbounded. If $a_1=-\infty$, $C_0=\emptyset$, and if
$b_k=+\infty$, $C_k=\emptyset$, in which cases it is an abuse to consider them connected components
of $E^\compl$, but they do not contribute to the terms in \cref{crit1d:res}.
Omitting the integrand $\rho_\eps(s-t)\ds\dt$ for the sake of readability, let us write
\begin{equation}\label{crit1d:eq1}
\Per_{\rho_\eps}^1(E)
=2\sum_{i=0}^k \int_{E} \int_{C_i}
=\sum_{i=0}^k \left(2\int_{C_i} \int_{\{t\,\in\,E~:~t>a_{i+1}\}} ~
+ 2\int_{C_i} \int_{\{t\,\in\,E~:~t<b_i\}}\right),
\end{equation}
with the convention $\{t<-\infty\}=\{t>+\infty\}=\emptyset$.
By \cref{eq:p1form1}, we have
\begin{equation}\label{crit1d:eq2}
P^1(E)=\sum_{i=1}^k \left(2\int_{-\infty}^{a_i}
\int_{a_i}^{+\infty}~+2\int_{-\infty}^{b_i}\int_{b_i}^{+\infty}\right).
\end{equation}
Note that this holds even if $a_1=-\infty$ or $b_k=+\infty$.
Let us define
\[
R_i\coloneqq 2\int_{-\infty}^{a_{i+1}} \int_{a_{i+1}}^{+\infty}~-2\int_{C_i}
\int_{\{t\,\in\,E~:~t>a_{i+1}\}},\qquad\forall i\in\{0,\dotsc,k\},
\]
and similarly
\[
L_i\coloneqq 2\int_{-\infty}^{b_i} \int_{b_i}^{+\infty}~-2\int_{C_i}
\int_{\{t\,\in\,E~:~t<b_i\}},\qquad\forall i\in\{0,\dotsc,k\}.
\]
Notice that $R_k=L_0=0$.
We observe that by definition of $L_i$, $R_i$ and \cref{crit1d:eq1}, \cref{crit1d:eq2},
\begin{equation}\label{crit1d:eq3}
\calE^1_{\rho_\eps}(E)
=\sum_{i=0}^{k} \left(L_i+R_i\right).
\end{equation}
Writing
\[
\int_{-\infty}^{a_{i+1}} \int_{a_{i+1}}^{+\infty}
= \int_{-\infty}^{b_i} \int_{a_{i+1}}^{+\infty}~
+\int_{C_i} \int_{a_{i+1}}^{+\infty},
\]
for $i\in\{0,\dotsc,k\}$, using \cref{lem:critenergy1} we have
\begin{equation}\label{crit1d:eq4}
R_i=2\left(\int_{-\infty}^{b_i}\int_{a_{i+1}}^{+\infty}~+
\int_{C_i}\int_{a_{i+1}}^{+\infty} \right) -2\int_{C_i}\int_{\{t\,\in\,E~:~t>a_{i+1}\}}
=2\int_{C_i} \int_{\{t\,\in\, E^\compl~:~t>a_{i+1}\}}~ +
\frac{1}{2}\calE^1_{\rho_\eps}(C_i).
\end{equation}
and similarly,
\begin{equation}\label{crit1d:eq5}
L_i=2\int_{C_i} \int_{\{t\,\in\, E^\compl~:~t<b_i\}}~ +
\frac{1}{2}\calE^1_{\rho_\eps}(C_i).
\end{equation}
The two previous equations hold even if $C_0$, $C_k$ are empty or unbounded.
Inserting \cref{crit1d:eq4,crit1d:eq5} into \cref{crit1d:eq3} yields
\[
\begin{aligned}
\calE^1_{\rho_\eps}(E)
&=2\sum_{i=0}^k
\left( \int_{C_i} \int_{E^\compl\,\cap\,\{t>a_{i+1}\}}~ +\int_{C_i}
\int_{E^\compl\,\cap\,\{t<b_i\}} \right) +\sum_{i=0}^k \calE^1_{\rho_\eps}(C_i)\\
&=2\sum_{i=0}^k \sum_{j\neq i} \int_{C_i} \int_{C_j}~ +\sum_{i=0}^k
\calE^1_{\rho_\eps}(C_i),
\end{aligned}
\]
which concludes the proof.
\end{proof}

We easily deduce from \cref{lem:crit1d} and the slicing decomposition of the critical energy stated
in \cref{cor:expecrit} that in dimension $n=2$, the critical energy of a connected set decreases by
convexification.

\begin{prp}\label{prp:decconvex}
If $E\subsq\R^2$ is a bounded, connected set of finite perimeter, then
\[
\calE_{G_\eps}(\co(E))\leq \calE_{G_\eps}(E),\qquad\forall \eps>0.
\]
\end{prp}

\begin{proof}
First, recall that a bounded convex set is a set of finite perimeter, since it is
Lebesgue-equivalent to an open set with Lipschitz boundary, thus $\co(E)$ is a set of finite
perimeter.
Then, by \cref{cor:expecrit} we have
\begin{equation}\label{decconvex:eq1}
\calE_{G_\eps}(E)
=\frac{1}{4} \int_{\bbS^1} \int_\R \calE^1_{\rho_\eps}(E_{\sigma,t}) \dt\dH^1_\sigma,
\quad\text{ and }\quad
\calE_{G_\eps}(\co(E))
=\frac{1}{4} \int_{\bbS^1} \int_\R \calE^1_{\rho_\eps}(F_{\sigma,t}) \dt\dH^1_\sigma,
\end{equation}
where
\[
E_{\sigma,t} := \Big\{ s\in\R~:~ t\sigma^\perp+s\sigma\in E\Big\},
\quad\text{ and }\quad
F_{\sigma,t} := \Big\{ s\in\R~:~ t\sigma^\perp+s\sigma\in \co(E)\Big\}.
\]
Since $E$ is connected, for every $\sigma\in\bbS^1$ and $t\in\R$, the slice
$F_{\sigma,t}$ is empty if and only if $E_{\sigma,t}$ is empty (this is the argument which is valid
in dimension $2$ only).
In addition, since $E$ and $\co(E)$ are bounded sets of finite perimeter in $\R^2$, for
$\calH^1$-almost every $\sigma\in\bbS^1$ and $\calL^1$-almost every $t\in\R$, $E_{\sigma,t}$ and
$F_{\sigma,t}$ are bounded sets of finite perimeter in $\R$, and for every nonempty slice,
$F_{\sigma,t}$ is an interval s.t. $E_{\sigma,t}\subsq F_{\sigma,t}$ (since $F_{\sigma,t}$ is a
slice of a convex set). Hence, by \cref{lem:critenergy1,lem:crit1d}, for $\calH^1$-almost every
$\sigma\in\bbS^1$ and $\calL^1$-almost every $t\in\R$, there holds
\[
\calE_{\rho_\eps}^1(F_{\sigma,t})\leq\calE_{\rho_\eps}^1(E_{\sigma,t}).
\]
In view of \cref{decconvex:eq1}, this concludes the proof.
\end{proof}

\subsection{Convexity of minimizers in 2D}\label{subsec:minconv}

In this part, we shall use for the sake of brevity the abbreviations
$\calF_{\gamma,\eps}:=\calF_{\gamma,G_\eps}$ and $\calE_{\eps}:=\calE_{G_\eps}$.

Consider a connected minimizer $E_\eps\subsq\R^2$ of \cref{minrp}. Recall that
$\calF_{\gamma,\eps}=(1-\gamma)P+\gamma\calE_\eps$, thus by~\cref{prp:decconvex},
$\calF_{\gamma,\eps}(\co(E_\eps))\leq\calF_{\gamma,\eps}(E_\eps)$.
However, if $E_\eps$ is not convex, $\abs{\co(E_\eps)}>\abs{E_\eps}$, so $\co(E_\eps)$ is not a
valid competitor for \cref{minrp}. Recalling that $B_{1-\delta(\eps)} \subsq E_\eps \subsq
B_{1+\delta(\eps)}$ for small $\eps$, we have
\[
B_{1-\delta(\eps)} \subsq \co(E_\eps) \subsq B_{1+\delta(\eps)},
\]
and defining $t_\eps\coloneqq \sqrt{\abs{B_1}/\abs{\co(E_\eps)}}$, we see that
$t_\eps\co(E_\eps)$ is a valid competitor with $(1+\delta(\eps))^{-1}<t_\eps<1$. Let us show that
$\calF_{\gamma,\eps}(t_\eps\co(E_\eps))<\calF_{\gamma,\eps}(\co(E_\eps))$. This follows from the
following result, which we prove further below.

\begin{lem}\label{lem:decconvscale}
There exists $\ol{\eps}_1=\ol{\eps}_1(G,\gamma)>0$ such that the following holds.
If $E\subsq\R^2$ is a convex set such that
\[
B_{1-\delta(\eps)} \subsq E \subsq B_{1+\delta(\eps)},
\]
with $0<\eps<\ol{\eps}_1$, and where $\delta$ is the function given by \cref{thm:existmin}, then
\[
\calF_{\gamma,\eps}(tE)< \calF_{\gamma,\eps}(E),\qquad\forall t\in({\textstyle\frac{1}{2}},1).
\]
\end{lem}

As a consequence, we obtain that minimizers of \cref{minrp} are necessarily convex for small $\eps$
in \mbox{dimension $n=2$}, that is \cref{mainthm:convexmin}:

\begin{proof}[Proof of {\cref{mainthm:convexmin}}]
Let $E_\eps\subsq\R^2$ be a minimizer of \cref{minrp} with $0<\eps<\eps_1$, \edit{where $\eps_1$ is
to be chosen later, smaller than $\eps_0(G,\gamma)$ and $\ol{\eps}_1(G,\gamma)$ given respectively
by \cref{thm:existmin} and \cref{lem:decconvscale}}. Thus, up to a Lebesgue-negligible set and a
translation, $E_\eps$ is connected and satisfies
\[
B_{1-\delta(\eps)}\subsq E_\eps\subsq \co(E_\eps)\subsq B_{1+\delta(\eps)}.
\]
By contradiction, let us assume that $E_\eps$ is not convex, so that $\abs{\co(E_\eps)}>\abs{B_1}$.
We have ${\abs{t_\eps\co(E_\eps)}=\abs{B_1}}$ with
$t_\eps\coloneqq\sqrt{\abs{B_1}/\abs{\co(E_\eps)}}$. \edit{Choosing $\eps_1=\eps_1(G,\gamma)$ small
enough, we have $1/2\leq(1+\delta(\eps))^{-1}<t_\eps<1$. Since} $\eps_1<\ol{\eps}_1(G,\gamma)$, by
\cref{lem:decconvscale}, we find
$\calE_{\gamma,\eps}(t_\eps\co(E_\eps))<\calE_{\gamma,\eps}(\co(E_\eps))$.
Now, since $E_\eps$ is connected, there holds
$\calE_{\gamma,\eps}(\co(E_\eps))\leq\calE_{\gamma,\eps}(E_\eps)$ by \cref{prp:decconvex}, hence
$\calE_{\gamma,\eps}(t_\eps\co(E_\eps))<\calE_{\gamma,\eps}(E_\eps)$. This contradicts the
minimality of $E_\eps$, whence $E_\eps$ is convex.
\end{proof}

\subsection{Proof of \texorpdfstring{{\cref{lem:decconvscale}}}{Lemma 3.8}}


In order to prove that $\calF_{\gamma,\eps}(\co(E_\eps))$ (strictly) decreases by scaling
down~$\co(E_\eps)$ by $t_\eps$ for small $\eps$, we need to estimate the term
$\wt{\Per}_{G_\eps}(t_\eps\co(E_\eps))$ appearing when applying \cref{lem:derivpg}.
While it is not so difficult to see that for a fixed set $E$ with $C^1$ boundary,
$\wt{\Per}_{G_\eps}(E)$ converges to $P(E)$ as $\eps$ vanishes, here we do not have uniform
regularity estimates on minimizers of~\cref{minrp}. \edit{The lack of such uniform regularity
estimates as $\eps$ goes to $0$ is the main obstacle for proving that the unit ball is, up to
translations, the unique minimizer for $\eps$ small enough. With sufficient regularity, the
arguments of~\cref{sec:fuglede} would yield the result.}
However, here we can make use of the convexity of $\co(E_\eps)$ and the fact that it lies between two
\edit{disks} whose radii are close to $1$ to prove the following.

\begin{lem}\label{lem:minalmflat}
For any $\alpha\in(0,1)$, there exist positive constants $r=r(\alpha)$ and
$\ol{\eps}_2=\ol{\eps}_2(G,\gamma,\alpha)$ such that the following holds. If $E\subsq\R^2$ is a
convex set such that
\[
B_{1-\delta(\eps)} \subsq E \subsq B_{1+\delta(\eps)},
\]
with $0<\eps<\ol{\eps}_2$, and where the function $\delta$ is given by \cref{thm:existmin}, then we
have
\begin{equation}\label{minalmflat:res1}
\left\{ y\in B_r(x)~:~ \frac{x-y}{\abs{x-y}}\cdot \nu_E(x)\geq \alpha\right\}
\subsq E,\qquad\text{ for }\calH^1\text{-almost every }x\in\partial E.
\end{equation}
\end{lem}

\begin{proof}
We proceed in two steps.\\[4pt]
\textit{Step 1.}
We show that the normal vector $\nu_E(x)$ converges uniformly for $\calH^1$-almost every
$x\in\partial E$ to $\nu_{B_1}(x)$ as $\eps$ vanishes, in the sense that
\begin{equation}\label{minalmflat:eq1}
\abs*{\nu_E(x)-\frac{x}{\abs{x}}}\leq 2\sqrt{\frac{\delta(\eps)}{1+\delta(\eps)}},
\qquad \text{ for }\calH^1\text{-almost every }x\in\partial E.
\end{equation}
The simple geometric argument is well-known (see e.g. \autocite{Oss1987,Fug1989}), but we state it
here for the reader's convenience.
Let $\theta_\eps\in (0,\pi)$ be the angular diameter of the \edit{disk} $B_{1-\delta(\eps)}$ from
any point $x\in\partial B_{1+\delta(\eps)}$, that is, $\theta_\eps$ is the non-oriented angle
between the two lines passing by $x\in \partial B_{1+\delta(\eps)}$ and tangent
to~$B_{1-\delta(\eps)}$.  Then we have
$\sin\left(\theta_\eps/2\right)=(1-\delta(\eps))/(1+\delta(\eps))$.
Let $x\in\partial E$, and let $\vphi\in (0,\pi/2)$ be the non-oriented angle between $x$ and
some tangent line to $E$ at $x$.
Since $E$ is convex and $B_{1-\delta(\eps)}\subsq E$, the tangent cone to $E$ at $x$ cannot
intersect $B_{1-\delta(\eps)}$, and since $\abs{x}\leq 1+\delta(\eps)$, we find
\[
\sin\vphi\geq\sin\left(\frac{\theta_\eps}{2}\right)\frac{1-\delta(\eps)}{1+\delta(\eps)}.
\]
By convexity of $E$, for $\calH^1$-almost every $x\in\partial E$, there is a unique tangent line to
$E$ at $x$, and writing
\[
\abs*{\nu_E(x)-\frac{x}{\abs{x}}}^2 = 2(1-\sin(\vphi))
=\frac{4\delta(\eps)}{1+\delta(\eps)}
\]
we deduce \cref{minalmflat:eq1}.\\[4pt]
\textit{Step 2.}
Let $\alpha\in (0,1)$, and let $r(\alpha)$, $\ol{\eps}_2(G,\gamma,\alpha)$ to be
fixed later. Then let $E\subsq\R^2$ be as in the statement of the lemma, with $0<\eps<\ol{\eps}_2$.
As a consequence of Step 1, we may assume that for any $\eps$ small enough depending only on $G$,
$\gamma$ and $\alpha$, we have
\[
\abs*{\frac{x-y}{\abs{x-y}}\cdot\left(\nu_E(x)-\frac{x}{\abs{x}}\right)}\leq\frac{\alpha}{2},
\qquad\text{ for }\calH^1\text{-almost every } x\in\partial E,\text{ and for every }y\in\Ren,
\]
so that proving \cref{minalmflat:res1} amounts to showing that for~$r(\alpha)$
and $\ol{\eps}_2(G,\gamma,\alpha)$ small enough,
\begin{equation}\label{minalmflat:eq2}
\left\{ y\in B_r(x)~:~ \frac{x-y}{\abs{x-y}}\cdot \frac{x}{\abs{x}}\geq \frac{\alpha}{2}\right\}
\subsq E,\qquad\forall x\in\partial E.
\end{equation}
We assume that $r\in(0,1/2)$ and then that $\ol{\eps}_2=\ol{\eps}_2(G,\gamma,r)\in (0,r/2)$
is such that $\delta(\eps)<1/4$ whenever $0<\eps<\ol{\eps}_2$. Let
$x=\abs{x}(\cos\psi,\sin\psi)\in\partial E$ with $\psi\in[0,2\pi)$. Since $x\in
B_{1+\delta(\eps)}\setminus B_{1-\delta(\eps)}$, the set~$\partial B_r(x)\cap B_{1-\delta(\eps)}$ is
made of exactly two points $A$ and $B$. Let us denote by $\vphi(x,r)\in(0,\pi)$ the angle between
the segments $[Ax]$ and $[Bx]$. Then let us introduce the circular sector $C(x,r)$ of $B_r(x)$
delimited by $A$ and $B$ (see \cref{fig:figdisksec}), that is
\begin{figure}
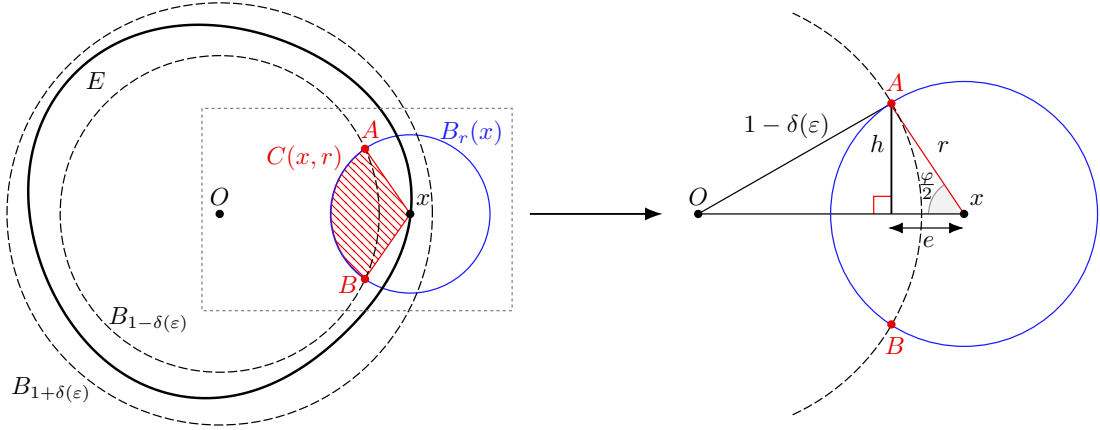

\centering
\includestandalone[width=.95\textwidth]{2d_figure1}
\caption{The situation in Step~2 of the proof of \cref{lem:minalmflat}. The red-dashed area is
$C(x,r)$.}
\label{fig:figdisksec}
\end{figure}
\begin{equation}\label{minalmflat:eq3}
\begin{aligned}
C(x,r)
\coloneqq& \left\{x-s\big(\cos(\theta+\psi),\sin(\theta+\psi)\big)~:~
0<s<r,~\abs{\theta}\leq\frac{\vphi(x,r)}{2}\right\}\\
=& \left\{y\in B_r(x)~:~ \frac{x-y}{\abs{x-y}}\cdot\frac{x}{\abs{x}}\geq
\cos\left(\frac{\vphi(x,r)}{2}\right) \right\}.
\end{aligned}
\end{equation}
We claim that for $\ol{\eps}_2$ small enough, $\sin\left(\vphi(x,r)/2\right)\geq
\sqrt{1-r^2}$.
The situation is as in \cref{fig:figdisksec}, where we introduce the lengths $h$ and $e$. We have
\[
\left\{\begin{array}{rl}
h^2+e^2 &= r^2\\
(\abs{x}-e)^2+h^2 &= \big(1-\delta(\eps)\big)^2
\end{array}\right.,
\]
which gives, after computation,
\[
h^2=r^2\left[1-\frac{1}{4\abs{x}^2}\left(r^2+2\big(\abs{x}^2-(1-\delta(\eps))^2\big)
+\frac{1}{r^2}\big(\abs{x}^2-(1-\delta(\eps))^2\big)^2 \right)\right].
\]
Recalling that $1-\delta(\eps)\leq\abs{x}<1+\delta(\eps)$ and $r<1/2$, this implies
\[
\begin{aligned}
h^2
\geq r^2-\frac{r^4}{4\big(1+\delta(\eps)\big)^2}\left[1+\frac{4\delta(\eps)}{r^2}
+\left(\frac{4\delta(\eps)}{r^2}\right)^2\right]
&\geq r^2\left(1-r^2\right),
\end{aligned}
\]
provided $\ol{\eps}_2=\ol{\eps}_2(G,\gamma,r)$ is chosen small enough.
Thus, the angle $\vphi(r,x)$ satisfies
\[
\sin\left(\frac{\vphi(r,x)}{2}\right)=\frac{h}{r}\geq \sqrt{1-r^2},
\]
which proves the claim.
Choosing $r=r(\alpha)$ and $\ol{\eps}_2(G,\gamma,r)$ small enough, we deduce
$\cos\left(\vphi(r,x)/2\right) \leq \frac{\alpha}{2}$. By convexity of $E$ and the fact that
$B_{1-\delta(\eps)}\subsq E$, we have $C(x,r)\subsq E$, hence \cref{minalmflat:eq2} holds, in view
of~\cref{minalmflat:eq3}. This concludes the proof.
\end{proof}

Then we can estimate $\wt{P}_{G_\eps}(E)$ when $E\subsq\R^2$ is a convex set lying between the disks
$B_{1-\delta(\eps)}$ and~$B_{1+\delta(\eps)}$, for any $\eps$ small enough.

\begin{lem}\label{lem:mincone}
For any $\tau>0$, there exists $\ol{\eps}_3=\ol{\eps}_3(G,\gamma,\tau)>0$ such that the following
holds. If $E\subsq\R^2$ is a convex set such that
\[
B_{1-\delta(\eps)} \subsq E \subsq B_{1+\delta(\eps)},
\]
with $0<\eps<\ol{\eps}_3$, where the $\delta$ function is the one from \cref{thm:existmin}, then we
have
\[
(1-\tau)P(E)\leq\wt{P}_{G_\eps}(E)\leq P(E),
\]
where $\wt{P}_{G_\eps}(E)$ is defined by \cref{defReps}.
\end{lem}

\begin{proof}
Let $E$ be as in the statement of the lemma, with $0<\eps<\ol{\eps}_3(G,\gamma,\tau)$, where
$\ol{\eps}_3$ is to be fixed later. Recall that since $E$ is a bounded convex set, it is a set of
finite perimeter and its topological boundary is $\calH^1$-equivalent to $\partial^* E$.
We proceed in two steps.\\[4pt]
\textit{Step 1. Upper bound.}
First, note that by convexity of $E$, we have
\[
(x-y)\cdot\nu_E(x)\geq 0,\qquad\text{ for }\calH^1\text{-almost every }x\in\partial E,\text{ and
every } y\in E,
\]
and, defining for $\calH^1$-almost every $x\in\partial E$,
\[
H_x:= \Big\{ y\in \Ren ~:~ (x-y)\cdot \nu_E(x)>0\Big\},
\]
it holds
\[
\wt{P}_{G_\eps}(E)\leq 2\int_{\partial E} \int_{H_x} G_\eps(x-y) (x-y)\cdot \nu_E(x)\dH_x^1\dy.
\]
Then, letting $g_\eps=\eps^{-3}g(\eps^{-1}\,\cdot)$, by a change of variable and the coarea
formula, we find
\[
\begin{aligned}
\wt{P}_{G_\eps}(E)
&\leq 2\int_{\partial E} \int_{\left\{z\,\in\,\R^2~:~z\cdot\nu_E(x)\,>\,0\right\}}
G_\eps(z)\big(z\cdot\nu_E(x)\big)\dz\dH_x^1\\
&=2\int_{\partial E}\int_0^{+\infty} g_\eps(t)t^2\left(\int_{\left\{\sigma\,\in\,\bbS^1~:~\sigma
\cdot\nu_E(x)\,>\,0\right\}} \sigma\cdot\nu_E(x)\dH^1_\sigma\right)\dt\dH^1_x\\
&=\int_{\partial E}\int_0^{+\infty} g_\eps(t)t^2\left(\int_{\bbS^{1}}
\abs{\sigma\cdot\nu_E(x)}\dH_\sigma^1\right)\dt\dH_x^1\\
&=\bfK_{1,2}\int_{\partial E}\abs{\bbS^1}\int_0^{+\infty} g_\eps(t)t^2\dt\dH_x^1
=I_G^1\bfK_{1,2}P(E)=P(E),
\end{aligned}
\]
where we also used the definition of $\bfK_{1,2}$ for the third equality, and \cref{Hint} for the
last one.\\[4pt]
\textit{Step 2. Lower bound.}
Let us first recall that by assumption \cref{Hint},
\begin{equation}\label{id_PE}
\begin{aligned}
P(E)
=4\int_{\partial E}\int_0^{+\infty} t^2g_\eps(t)\dt\dH^1
=2\int_{\partial E} \int_0^{+\infty} t^2g_\eps(t)\int_{\left\{\sigma\,\in\,\bbS^1~:~\sigma\cdot
e\,>\,0\right\}} (\sigma\cdot e) \dH^1_\sigma\dt\dH^1,
\end{aligned}
\end{equation}
for all $ e\in\bbS^1$.
Let $\alpha\in(0,1)$ to be chosen later, and $\theta_0\coloneqq \arcsin(\alpha)$.
Then let $r=r(\alpha)>0$ and~$\ol{\eps}_2(G,\gamma,\alpha)$ given by \cref{lem:minalmflat}. Assume
that $\ol{\eps}_3<\ol{\eps}_2$.
Let us write
\[
\begin{aligned}
\wt{P}_{G_\eps}(E)
&=2\int_{\partial E} \int_{E\,\cap\,\{\abs{x-y}\,<\,r\}} G_\eps(x-y)\,(x-y)\cdot\nu_E(x)\dy\dH^1_x\\
&\hphantom{=\int_{\partial E} \int_E}
+2\int_{\partial E} \int_{E\,\cap\,\{\abs{x-y}\,\geq\, r\}}
G_\eps(x-y)\,(x-y)\cdot\nu_E(x)\dx\dH^1_y \eqqcolon
\wt{P}_{G_\eps}^{(1)}(E)+\wt{P}_{G_\eps}^{(2)}(E).
\end{aligned}
\]
Since $\eps<\ol{\eps}_2$, by \cref{lem:minalmflat}, we have
\[
\begin{aligned}
\wt{P}_{G_\eps}^{(1)}(E)
&\geq 2\int_{\partial E} \int_{\left\{y\,\in\, B_r(x)~:~
\frac{x-y}{\abs{x-y}}\cdot\nu_E(x)\,\geq\,\alpha\right\}}
G_\eps(x-y)(x-y)\cdot\nu_E(x)\dy\dH_x^1\\
&=2\int_{\partial E} \int_0^{+\infty} g_\eps(t)t^2
\int_{\left\{\sigma\,\in\,\bbS^1~:~\alpha\,\leq\,\sigma\cdot
\nu_E(x)\,\leq\, 1\right\}} \sigma\cdot\nu_E(x)\dH^1_\sigma\dt\dH^1_x.
\end{aligned}
\]
We compute
\[
\int_{\{\sigma\,\in\,\bbS^1~:~\alpha\,\leq\,\sigma\cdot\nu_E(x)\,\leq\, 1\}}
\sigma\cdot\nu_E(x)\dH^1_\sigma
=2\int_0^{\arccos\alpha} \cos\theta\dd\theta
=2\sin(\arccos\alpha)
=\sqrt{1-\alpha^2}.
\]
Together with \cref{id_PE}, this leads do
\[
\begin{aligned}
\wt{P}_{G_\eps}^{(1)}(E)
\geq \sqrt{1-\alpha^2}P(E).
\end{aligned}
\]
Choosing $\alpha$ small enough, depending only on $\tau$, we then find
\[
\wt{P}_{G_\eps}^{(1)}(E)\geq \left(1-\frac{\tau}{2}\right)P(E).
\]
Since $\alpha=\alpha(\tau)$, we have $r=r(\tau)$, and $\ol{\eps}_2=\ol{\eps}_2(G,\gamma,\tau)$.
Now with $r$ fixed, up to choosing $\ol{\eps}_3<\ol{\eps}_2$ even smaller if needed, there
holds
\[
\abs{\wt{P}_{G_\eps}^{(2)}(E)} \leq P(E)\int_{B_r^\compl} \abs{x}G_\eps(x)\dx
= P(E)\int_{B_{r/\eps}^\compl} \abs{x}G(x)\dx \leq \frac{\tau}{2}P(E),
\]
since the first moment of $G$ is finite. Hence, from our choice of $r$ and $\ol{\eps}_3$, we obtain
\[
\wt{P}_{G_\eps}(E)\geq (1-\tau)P(E),
\]
which concludes the proof.
\end{proof}


We are now in position to prove \cref{lem:decconvscale}.

\begin{proof}[Proof of {\cref{lem:decconvscale}}]
\edit{Let $\tau>0$ to be chosen later, and let $\ol{\eps}_1=\ol{\eps}_3(G,\gamma,\tau)/4$, where
$\ol{\eps}_3$ is given by \cref{lem:mincone}. Assume that $0<\eps<\ol{\eps}_1$.}
Recall that by \cref{lem:derivpg}, for almost every $t$, we have
\begin{equation}\label{decconvscale:eq1}
\begin{aligned}
\frac{\dd}{\dt} \left[\Per_{G_\eps}(tE)\right]
&= \frac{2}{t}\Per_{G_\eps}(tE)-\frac{1}{t}\wt{P}_{G_\eps}(tE),
\end{aligned}
\end{equation}
where $\wt{P}_{G_\eps}$ is defined by \cref{defReps}.
Changing variables, we deduce
\[
\begin{aligned}
\frac{1}{t}\wt{P}_{G_\eps}(tE)
&=2t^3\int_E\int_{\partial E} G_\eps(t(x-y))(y-x)\cdot\nu_E(y)\\
&=2\int_E\int_{\partial E} G_{(t^{-1}\eps)}(x-y)(y-x)\cdot\nu_E(y) \dH_y^1\dx
=\wt{P}_{G_{t^{-1}\eps}}(E).
\end{aligned}
\]
\edit{Let $\edit{\tau\coloneqq(1-\gamma)/(2\gamma)}$. By our choice of $\ol{\eps}_1$,
$t^{-1}\eps<\ol{\eps}_3$, thus by~\cref{lem:mincone}, we have}
\[
\frac{1}{t}\wt{P}_{G_\eps}(tE)\geq (1-\tau)P(E).
\]
With \cref{decconvscale:eq1} and \cref{prp:boundPerG}, this leads to
\[
\frac{\dd}{\dt} \left[\Per_{G_\eps}(tE)\right]
\leq \frac{2}{t}P(tE)-(1-\tau)P(E)
= (1+\tau)P(E).
\]
Hence, for every $t\in(1/2,1)$, by our choice of $\tau$, it follows
\[
\calF_{\gamma,\eps}(E)-\calF_{\gamma,\eps}(tE)
=(1-t) P(E)-\gamma\int_t^1 \frac{\dd}{\dd s} \left[\Per_{G_\eps}(sE)\right] \dd s
\geq (1-t)P(E)\big[1-\gamma(1+\tau)\big]>0.
\]
This proves the lemma.
\end{proof}

\subsection{Convex minimizers are nearly spherical sets}\label{subsec:convnearlysph}

In arbitrary dimension, it is classical to improve Hausdorff convergence of the boundary of
minimizers to Lipschitz convergence once convexity is established.
Here, we show that convex minimizers of \cref{minrp} are centered $t(\eps)$-nearly spherical sets
(see \cref{maindfn:nearlysph}), up to a translation, where the function $t$ vanishes in $0^+$.

\begin{prp}[Convex minimizers are nearly spherical sets]\label{prp:minsph}
There exists $\ol{\eps}_4=\ol{\eps}_4(n,G,\gamma)>0$ such that the following holds.
If $E_\eps\subsq\Ren$ is a convex minimizer of \cref{minrp} with $0<\eps<\ol{\eps}_4$, then, up to
a translation, we have
\[
\partial E_\eps=\Big\{ (1+u_\eps(x))x ~:~ x\in\bbS^{n-1} \Big\}
\]
and
\[
\int_{E_\eps} x\dx = 0,
\]
with $u_\eps\in \Lip(\bbS^{n-1})$, $\norm{u_\eps}_{L^\infty(\bbS^{n-1})}\leq C'\delta(\eps)$ and
$\norm{\grad_\tau\, u_\eps}_{L^\infty(\bbS^{n-1})}\leq C''\delta(\eps)^{1/2}$, where
$\delta=\delta(n,G,\gamma)$ is the function of \cref{thm:existmin} vanishing in $0^+$, and
$C',C''>0$ only depend on $n$.
\end{prp}

\begin{proof}
In the proof, we write $\norm{\cdot}_\infty$ for $\norm{\cdot}_{L^\infty(\bbS^{n-1})}$.
Let $E_\eps$ be a convex minimizer \cref{minrp} with $0<\eps<\ol{\eps}_4$, where $\ol{\eps}_4$ is to
be fixed later. If $\ol{\eps}_4<\eps_0$, where $\eps_0$ is given by \cref{thm:existmin}, up to a
translation and a negligible set, $E_\eps$ lies between the balls $B_{1-\delta(\eps)}$ and
$B_{1+\delta(\eps)}$. By convexity, this implies that the set $E_\eps$ itself (without the addition
or the subtraction of a negligible set) satisfies, up to a translation,
\[
B_{1-\delta(\eps)}\subsq E_\eps \subsq \ol{B}_{1+\delta(\eps)}.
\]
Setting $\displaystyle y_\eps:=-\frac{1}{\abs{B_1}}\int_{E_\eps}x\dx$, we have
\[
\int_{y_\eps+E_\eps} x\dx
=y_\eps\abs{E_\eps}+\int_{E_\eps} x\dx=y_\eps\abs{B_1}+\int_{E_\eps}x\dx=0.
\]
Notice that
\[
-\int_{E_\eps} x\dx=-\int_{B_{1+\delta(\eps)}}x\dx+\int_{B_{1+\delta(\eps)}\setminus E_\eps} x\dx
=\int_{B_{1+\delta(\eps)}\setminus E_\eps} x\dx
\]
thus
\[
\abs*{\int_{E_\eps} x\dx}\leq (1+\delta(\eps))\abs{B_{1+\delta(\eps)}\setminus E_\eps}
\leq \big(1+\delta(\eps)\big)\left[\big(1+\delta(\eps)\big)^n-\big(1-\delta(\eps)\big)^n\right]
\abs{B_1}
\leq C(n)\delta(\eps)
\]
provided that $\ol{\eps}_4$ is small enough depending only on $n$, $G$ and $\gamma$. Hence, up to
translating $E_\eps$ by $y_\eps$, we may assume that it satisfies
\begin{equation}\label{minspher:eq1}
B_{1-C'\delta(\eps)}\subsq E_\eps \subsq B_{1+C'\delta(\eps)},
\end{equation}
with $C'\coloneqq C(n)+1$ and is centered, that is,
\[
\int_{E_\eps} x\dx=0.
\]
By convexity of $E_\eps$, for every $x\in\bbS^{n-1}$, there is a unique point of intersection
$p_\eps(x)=t_\eps(x) x$ of~${\{tx~:~t>0\}}$ and $\partial E_\eps$, and by \cref{minspher:eq1},
$\abs{p_\eps(x)-x}\leq C'\delta(\eps)$. The map $p_\eps: \bbS^{n-1}\to\partial E_\eps$ is obviously
onto, so that, setting $u_\eps(x)=t_\eps(x)-1$, we have
\[
\partial E_\eps=\Big\{ (1+u_\eps(x))x ~:~ x\in\bbS^{n-1} \Big\},
\]
and $\norm{u_\eps}_\infty\leq C'\delta(\eps)$. In addition, the fact that $E_\eps$ is convex implies
$u_\eps\in \Lip(\bbS^{n-1})$. Moreover, for any~$x\in\bbS^{n-1}$, the distance between the normal
vector of $\partial E_\eps$ (which exists for $\calH^{n-1}$-almost every $x\in\partial E_\eps$) at
$p_\eps(x)$ and the vector $x$ (which is the normal vector of $\bbS^{n-1}$ at $x$) is controlled
by~$\norm{p_\eps-\id}_{\infty}=\norm{u_\eps}_{\infty}$, in view of Step 1 of the proof of
\cref{lem:minalmflat} (or simply by \autocite[Corollary~1]{Oss1987}), which gives a control of
$\norm{\grad_\tau\, u_\eps}_{\infty}$ by $\norm{u}_{\infty}$.  More precisely, by
\autocite[Inequality $(**)$]{Fug1989}, we have
\[
\norm{\grad_\tau\, u_\eps}_{\infty}
\leq 2\left(\frac{1+\norm{u_\eps}_{\infty}}{1-\norm{u_\eps}_{\infty}}\right)
\norm{u_\eps}_{\infty}^{\frac{1}{2}} \leq C''\delta(\eps)^{\frac{1}{2}},
\]
where we used the inequality $\norm{u_\eps}_\infty\leq C'\delta(\eps)\leq 1/2$ for the last
inequality, provided that $\ol{\eps}_4(n,G,\gamma)$ is chosen small, and defined $C''\coloneqq
6\sqrt{C'}$. This concludes the proof.
\end{proof}

Observe that \cref{mainthm:minnearlycirc} is an immediate corollary of
\cref{mainthm:convexmin,prp:minsph}.

\begin{rmk}\label{rmk:depeps12}
\edit{Let us comment on the dependence of the constants $\eps_1$ and $\eps_2$ respectively of
\cref{mainthm:convexmin} and \cref{mainthm:minnearlycirc} w.r.t. the parameter $\gamma$. Tracking
the constants, $\tau=(1-\gamma)/(2\gamma)$ in the proof of \cref{lem:decconvscale}, and
$\alpha=C \min(1,\sqrt{\tau})$ in the proof of \cref{lem:mincone}, for some universal constant $C$.
In \cref{lem:minalmflat}, $r=C\min(1,\alpha)$, and $\ol{\eps}_2$ is chosen so that $\delta(\eps)\leq
C\min(1,r^2)$ for every $0<\eps<\ol{\eps}_2$.  Combining all these with the asymptotic behavior of
$\delta(\eps)$ given in \cref{rmk:distball}, we deduce that we can choose
\[
\ol{\eps}_2=C(n,G)\min\Big(1,\Big(\frac{1-\gamma}{\gamma}\Big)^{\frac{3}{2}}\Big).
\]
Then, $\ol{\eps}_3$ of \cref{lem:mincone} must be smaller than $\ol{\eps}_2$ and satisfy
\[
\int_{B_{r/\ol{\eps}_2^\compl}} \abs{x}G(x)\dx\leq \frac{\tau}{2}.
\]
Assuming that $G(x)=O(\abs{x}^{-(n+1+\beta)})$ at infinity for some positive constant $\beta$, and
using that $r$ is chosen to be $C\min\big(1,\sqrt{(1-\gamma)/\gamma}\big)$,
we can choose
\[
\ol{\eps}_3=C(n,G)\min\Big(1,\Big(\frac{1-\gamma}{\gamma}\Big)^{\max\left(\frac{3}{2},\frac{1}{2}
+\frac{1}{\beta}\right)}\Big).
\]
In \cref{lem:decconvscale}, $\ol{\eps}_1=\ol{\eps}_3/4$, and in \cref{mainthm:convexmin},
$\eps_1$ is simply chosen smaller than $\ol{\eps}_1$. Notice that in \cref{prp:minsph},
$\ol{\eps}_4=C(n,G)\sqrt{(1-\gamma)/\gamma}$ in view of
\cref{rmk:distball}, and that $\eps_2$ of \cref{mainthm:minnearlycirc} must be smaller than
$\min(\eps_1,\ol{\eps}_4)$ and such that $\delta(\eps)^{1/2}< C(n)$. In the end, $\eps_1$ and
$\eps_2$ can be chosen to be}
\[
\edit{\eps_i=C(n,G)\min\Big(1,\Big(\frac{1-\gamma}{\gamma}\Big)^{\max\left(\frac{3}{2},\frac{1}{2}
+\frac{1}{\beta}\right)}\Big),\quad i\in\{1,2\}.}
\]
\end{rmk}

\section{Minimality of the unit ball among nearly spherical sets}\label{sec:fuglede}

This section is devoted to the proof of the minimality of the unit ball of $\Ren$ among $t$-nearly
spherical sets, for small $t$ and $\eps$.

\subsection{A Fuglede-type result for the nonlocal perimeter}

\edit{For a centered nearly spherical set $E_t$, we wish to estimate the quantity
$\calF_{\gamma,\eps}(E_t)-\calF_{\gamma,\eps}(B_1)$ from below. Thus, we need an estimate of
$P(E_t)-P(B_1)$ from below, and an estimate of $P_\eps(E_t)-P(B_1)$ from above. 
For the former we use the following lower bound.}

\begin{lem}\label{lem:fugledeper}
\edit{There exist positive constants $\ol{t}=\ol{t}(n)$ and $C=C(n)$ such that the following holds.
For any centered $t$-nearly spherical set $E_t$ with $0<t<\ol{t}$ described by the function $u$, we
have}
\[
\edit{P(E_t)\geq
P(B_1)+\frac{t^2}{2}\left(\norm{\grad_\tau\,u}_{L^2(\bbS^{n-1})}^2
-(n-1)\norm{u}_{L^2(\bbS^{n-1})}^2\right) -Ct^3\left(\norm{\grad_\tau\,
u}_{L^2(\bbS^{n-1})}^2+\norm{u}_{L^2(\bbS^{n-1})}^2\right).}
\]
\end{lem}

\edit{This is an improvement of the original lower bound of Fuglede \autocite[Theorem~1.2]{Fug1989}.
It can be found in \autocite[Proof of Theorem~3.1]{Fus2015}.\\
We establish a similar Taylor expansion for the nonlocal perimeter $\Per_{G_\eps}$. Let us point out
that similarly to fractional perimeters, fractional Sobolev-type seminorms associated with the
kernel $G_\eps$ (and its derivatives) naturally appear in the expansion.  However, these converge to
the $H^1$ seminorm as $\eps$ vanishes, so we choose here to control
$\Per_{G_\eps}(E_t)-\Per_{G_\eps}(B_1)$ directly in terms of the $H^1$ seminorm of $u$ on
$\bbS^{n-1}.$}

\begin{lem}\label{lem:fuglede}
Assume that $G$ satisfies \cref{Hrad,,Hint,,Hhighermoments}.
There exist positive constants $\ol{\eps}_5=\ol{\eps}_5(n,G)$ and $t_0=t_0(n)$ such that the
following holds.
If $E_t$ is a centered $t$-nearly spherical set with $0<t<t_0$ described by the function $u$, then
for any $0<\eps<\ol{\eps}_5$, we have
\begin{equation}\label{eq:resfuglede}
\begin{aligned}
\Per_{G_\eps}(E_t)
&\leq \Per_{G_\eps}(B_1)+\frac{t^2}{2}
\left(\big(1+Cq_\eta(\eps)\big)\norm{\grad_\tau\, u}_{L^2(\bbS^{n-1})}^2
-(n-1)\norm{u}_{L^2(\bbS^{n-1})}^2\right)\\
&\hphantom{\leq\Per{G_\eps}(B)+\frac{t^2}{2}\left(\big(1+CQ(\eta_\eps)\big)\right.}
+Ct^3\left(\norm{\grad_\tau\, u}_{L^2(\bbS^{n-1})}^2+\norm{u}_{L^2(\bbS^{n-1})}^2\right),
\end{aligned}
\end{equation}
where $C=C(n,G)$, $q_\eta$ is given by \cref{lem:cvsphunif}, and $q_\eta(\eps)\to 0$ as $\eps\to 0$.
\end{lem}

\edit{When $\eps$ vanishes, the quadratic terms from the Taylor expansion of
$\Per_{G_\eps}(E_t)-\Per_{G_\eps}(B_1)$ compensate exactly those of $P(B_1)-P(E_t)$, so that
the constant $\gamma$ must be smaller than one, and the expansion needs to be pushed to the third
order to obtain a useful estimate.}

\begin{proof}
In the proof, unless stated otherwise, $C$ denotes a positive constant depending \textsl{only on
$n$} and possibly changing from line to line.
For the sake of brevity we write $B$ for the open unit ball in $\Ren$, $\partial B$ for the unit
sphere $\bbS^{n-1}$, and $\norm{\cdot}_{\infty}$ for $\norm{\cdot}_{L^\infty(\bbS^{n-1})}$.
Let $0<\eps<\ol{\eps}_5$, where $\ol{\eps}_5=\ol{\eps}_5(n,G)$ is to be fixed later.\\
Let $t_0=t_0(n)<1/8$ to be fixed later, and let $E_t$ be a centered $t$-nearly spherical set
such that $\partial E_t=\left\{(1+tu(x))x~:~x\in\bbS^{n-1}\right\}$ with $0<t<t_0$. We proceed in
three steps.\\[4pt]
\textit{Step 1.} We rewrite $\Per_{G_\eps}(E_t)$ in a more convenient form,
introducing two terms that we will control from above in the next steps. Using polar coordinates, we
have
\[
\Per_{G_\eps}(E_t)=2\iint_{\partial B\times\partial B} \int_0^{1+tu(x)} \int_{1+tu(y)}^{+\infty}
G_\eps(rx-\rho y)r^{n-1}\rho^{n-1}\dr\dd\rho\dH^{n-1}_x\dH^{n-1}_y.
\]
By symmetry of $G$, we see that
\[
\begin{aligned}
\Per_{G_\eps}(E_t)&=\iint_{\partial B\times\partial B} \left(
\int_0^{1+tu(x)} \int_{1+tu(y)}^{+\infty} G_\eps(rx-\rho y)r^{n-1}\rho^{n-1}\dr\dd\rho\right.\\
&\hphantom{=\iint}
\left.+\int_0^{1+tu(y)} \int_{1+tu(x)}^{+\infty} G_\eps(rx-\rho y)r^{n-1}\rho^{n-1}\dr\dd\rho
\right)\dH^{n-1}_x\dH^{n-1}_y.
\end{aligned}
\]
Using the identity
\[
\int_0^b \int_a^{+\infty}~+\int_0^a \int_b^{+\infty}=\int_a^b \int_a^b~+\int_0^a \int_a^{+\infty}~
+\int_0^b\int_b^{+\infty},
\]
and the symmetry of $G$ yet again, we find
\[
\begin{aligned}
\Per_{G_\eps}(E_t)&=\iint_{\partial B\times\partial B} 
\int_{1+tu(y)}^{1+tu(x)} \int_{1+tu(y)}^{1+tu(x)} G_\eps(rx-\rho
y)r^{n-1}\rho^{n-1}\dr\dd\rho\dH^{n-1}_x\dH^{n-1}_y\\
&\hphantom{=\iint}
+2\iint_{\partial B\times\partial B}\int_0^{1+tu(x)} \int_{1+tu(x)}^{+\infty} G_\eps(rx-\rho
y)r^{n-1}\rho^{n-1}\dr\dd\rho\dH^{n-1}_x\dH^{n-1}_y.
\end{aligned}
\]
Changing variables, we deduce
\begin{equation}\label{fuglede:eq1}
\begin{aligned}
&\Per_{G_\eps}(E_t)\\
&\hphantom{\Per}=t^2\iint_{\partial B\times\partial B} \int_{u(y)}^{u(x)}\int_{u(y)}^{u(x)}
(1+ta)^{n-1}(1+tb)^{n-1}G_\eps\big(x-y+t(ax-by)\big)\da\db\dH_x^{n-1}\dH_y^{n-1}\\
&\hphantom{\Per}\hphantom{=t^2\iint_{\partial B\times}}
+2\iint_{\partial B\times\partial B}\int_0^{1+tu(x)} \int_{1+tu(x)}^{+\infty} G_\eps(rx-\rho
y)r^{n-1}\rho^{n-1}\dr\dd\rho\dH^{n-1}_x\dH^{n-1}_y\\
&\edit{\hphantom{\Per}\eqqcolon t^2 \vphi_\eps(t)+\psi_\eps(t).}
\end{aligned}
\end{equation}
This concludes Step 1.\\[4pt]
\textit{Step 2. Estimation of $\vphi_\eps(t)$.}
Note that by \cref{lem:cvsphunif,rmk:condTHinfty}, we have
\begin{equation}\label{fuglede:eq2}
\begin{aligned}
\vphi_\eps(0)
&=\iint_{\partial B\times\partial B} (u(x)-u(y))G_\eps(x-y)\dH_x^{n-1}\dH_y^{n-1}\\
&=\iint_{\partial B\times\partial B} \frac{(u(x)-u(y))^2}{\abs{x-y}^2}\eta_\eps(\abs{x-y})
\dH_x^{n-1}\dH_y^{n-1}\\
&\leq \big(1+Cq_\eta(\eps)\big)\norm{\grad_\tau\,u}_{L^2(\bbS^{n-1})}^2.
\end{aligned}
\end{equation}
Then let us write,
\begin{equation}\label{fuglede:eq3}
\begin{aligned}
&\vphi_\eps(t)-\vphi_\eps(0)\\
&=\iint_{\partial B\times\partial B} \int_{u(y)}^{u(x)}\int_{u(y)}^{u(x)}
\Big((1+ta)^{n-1}(1+tb)^{n-1}G_\eps\big(x-y+t(ax-by)\big)\\
&\hphantom{\iint_{\partial B\times\partial B} \int_{u(y)}^{u(x)}\int_{u(y)}^{u(x)}
\Big((1+ta)^{n-1}(1+tb)^{n-1}}-G_\eps(x-y)\Big)\da\db\dH_x^{n-1}\dH_y^{n-1}\\
&
=\iint_{\partial B\times\partial B} \int_{u(y)}^{u(x)}\int_{u(y)}^{u(x)}
\Big((1+ta)^{n-1}(1+tb)^{n-1}-1\Big)G_\eps\big(x-y\big)\da\db\dH_x^{n-1}\dH_y^{n-1}\\
&\hphantom{=\iint}
+\iint_{\partial B\times\partial B} \int_{u(y)}^{u(x)}\int_{u(y)}^{u(x)}
(1+ta)^{n-1}(1+tb)^{n-1}\Big(G_\eps\big(x-y+t(ax-by)\big)-G_\eps(x-y)\Big)\\
&\hphantom{=\iint}\hphantom{+\iint_{\partial B\times\partial B}
\int_{u(y)}^{u(x)}\int_{u(y)}^{u(x)}(1+ta)^{n-1}(1+tb)^{n-1}\Big(G_\eps(x-y}
\da\db\dH^{n-1}_x\dH^{n-1}_y\\
&\eqqcolon I_1(t)+I_2(t).
\end{aligned}
\end{equation}
On one hand, since $\norm{u}_\infty\leq 1$, on the domain of integration we have
\[
\abs*{(1+ta)^{n-1}(1+tb)^{n-1}-1}
\leq Ct\norm{u}_{\infty}
\leq Ct.
\]
Thus
\[
\begin{aligned}
I_1(t)
&\leq Ct\iint_{\partial B\times\partial B} \int_{u(x)}^{u(y)}\int_{u(x)}^{u(y)}
G_\eps\big(x-y\big)\da\db\dH_x^{n-1}\dH_y^{n-1}\\
&= Ct\iint_{\partial B\times\partial B} (u(x)-u(y))^2 G_{\eps}(x-y)\dH_x^{n-1}\dH_y^{n-1}\\
&= Ct\iint_{\partial B\times\partial B} \frac{(u(x)-u(y))^2}{\abs{x-y}^2}
\eta_{\eps}(\abs{x-y})\dH_x^{n-1}\dH_y^{n-1},
\end{aligned}
\]
where $\eta_\eps$ is defined as in \cref{lem:cvsphunif} so that
$g_\eps(r)=\eta_\eps(r)/(2r^2)$.
In view of \cref{lem:cvsphunif,rmk:condTHinfty} again, it follows
\begin{equation}\label{fuglede:eq5}
\begin{aligned}
I_1(t)
\leq Ct \big(1+q_\eta(\eps)\big)\int_{\partial B} \abs{\grad_\tau\, u}^2\dH^{n-1}
\leq Ct \norm{\grad_\tau\,u}_{L^2(\bbS^{n-1})}^2,
\end{aligned}
\end{equation}
for any $\eps$ small enough (depending only on $n$ and $G$).\\
Let us now bound the term $I_2(t)$. Integrating on a line (recall that $g$ is absolutely continuous
in $(0,+\infty)$ by \cref{Hhighermoments}), and using the inequality
$(1+ta)^{n-1}(1+tb)^{n-1}\leq 2^{2(n-1)}$ for any $0\leq \abs{a},\abs{b}\leq \norm{u}_\infty\leq 1$,
since $t<t_0<1$, we find
\begin{equation}\label{fuglede:eq3b}
I_2(t)\leq Ct\hspace{-2pt}\iint\limits_{\partial B\times\partial B}\hspace{-3pt}
\int_{u(x)}^{u(y)}\hspace{-3pt}\int_{u(x)}^{u(y)}\hspace{-3pt} \int_0^1\abs{\grad
G_\eps(x-y+st(ax-by))\cdot(ax-by)} \dd s\da\db\dH_x^{n-1}\dH_y^{n-1}.
\end{equation}
Observe that in \cref{fuglede:eq3b}, if the term $(x-y)$ inside $\grad G_\eps$ were not
perturbed by $st(ax-by)$, using the inequality
\begin{equation}\label{fuglede:eq3c}
\abs{ax-by}^2 = (a-b)^2+ab\abs{x-y}^2\leq \left(1+\norm{\grad_\tau\,
u}_\infty^2\right)\abs{x-y}^2\leq
2\abs{x-y}^2
\end{equation}
on the domain of integration, we would have
\[
I_2(t)\leq Ct\iint_{\partial B\times\partial B}
(u(x)-u(y))^2\abs{\grad G_\eps(x-y)}\abs{x-y} \dd s\dH_x^{n-1}\dH_y^{n-1},
\]
which we could estimate in terms of $\norm{\grad_\tau\, u}^2_{L^2(\bbS^{n-1})}$ by applying directly
\cref{lem:cvsphunif} with $t\mapsto t\abs{g'(t)}$ in place of $g$.
Here, to deal with this perturbation, we apply the technical \cref{lem:techlem1}, by showing 
that the right-hand side of \cref{fuglede:eq3b} is bounded by a term of the form
\[
Ct\int_0^1 \iint_{\partial B\times\partial B}\int_0^1\int_0^1 (u(x)-u(y))^2
k_\eps(\abs{\Phi_{stu}(x,y,r,\rho)})\dr\dd\rho\dH^{n-1}_x\dH^{n-1}_y\ds,
\]
where
\begin{equation}\label{fuglede:defPhi}
\Phi_{stu}(x,y,r,\rho)\coloneqq(x-y)+st\Big[\big(ru(x)+(1-r)u(y))\big)x-
\big(\rho u(y)+(1-\rho)u(x)\big)y\Big]
\end{equation}
is a small perturbation of $(x-y)$, and $k_\eps$ is a kernel defined further below.
Let us remark that, since~$G$ is radial, we have $\grad
G_\eps(x)=g_\eps'(\abs{x})x/\abs{x}$, thus
\begin{equation}\label{fuglede:eq4b}
\grad G_\eps(x-y+st(ax-by))\cdot(ax-by)
=g_\eps'(\abs{x-y+st(ax-by)})\frac{\frac{(a+b)}{2}\abs{x-y}^2+st\abs{ax-by}^2}{\abs{x-y+st(ax-by)}}.
\end{equation}
Then notice that, on the domain of integration, we have
\begin{equation}\label{fuglede:eq4a}
\begin{aligned}
\abs{x-y+st(ax-by)}^2
&=s^2t^2(a-b)^2+(1+sta)(1+stb)\abs{x-y}^2\\
&\geq (1+sta)(1+stb)\abs{x-y}^2-s^2t^2(u(x)-u(y))^2\\
&\geq \frac{9}{16}\abs{x-y}^2-\frac{1}{16}\norm{\grad_\tau\, u}_\infty^2\abs{x-y}^2
\geq \frac{1}{2}\abs{x-y}^2,
\end{aligned}
\end{equation}
since $t<t_0<1/4$.
By \cref{fuglede:eq3c}, \cref{fuglede:eq4a}, and the fact that $t<t_0<1/8$, on the domain of
integration we deduce
\[
\frac{(a+b)}{2}\abs{x-y}^2+st\abs{ax-by}^2
\leq \abs{x-y}^2+\frac{1}{4}\abs{x-y}^2
\leq 4\abs{x-y+st(ax-by)}^2.
\]
Hence, with \cref{fuglede:eq4b}, it follows
\begin{equation}\label{fuglede:eq4d}
\abs{\grad G_\eps(x-y+st(ax-by))\cdot(ax-by)} \leq C\, k_\eps(\abs{x-y+st(ax-by)}),
\end{equation}
where we have set $k(r)\coloneqq r\abs{g'(r)}$ and $k_\eps(r)\coloneqq
\eps^{-(n+1)}k(\eps^{-1}r)$ for all $r>0$ and $\eps>0$.
For $x,y\in\partial B$ fixed, making the changes of variables $a=u(y)+r(u(x)-u(y))$, and
$b=u(x)+\rho(u(y)-u(x))$ in \cref{fuglede:eq3b} yields, with~\cref{fuglede:defPhi,fuglede:eq4d},
\[
I_2(t)
\leq Ct \int_0^1 \iint_{\partial B\times\partial B}\int_0^1\int_0^1 (u(x)-u(y))^2
k_\eps(\abs{\Phi_{stu}(x,y,r,\rho)})\dr\dd\rho\dH^{n-1}_x\dH^{n-1}_y\ds.
\]
By \cref{Hhighermoments}, $I_{k}^1=I_G^2<\infty$ and $k(r)=O(r^{-(n+1)})$ at infinity, so the family
$(k_\eps)_{\eps>0}$ satisfies the assumptions of~\cref{lem:techlem1}. Thus choosing $t_0<t_1(n)$ and
$\ol{\eps}_5<\ol{\eps}_6$, where $t_1(n)$ and $\ol{\eps}_6(n,G)$ are given by \cref{lem:techlem1},
we have
\[
I_2(t)
\leq Ct\norm{\grad_\tau\, u}^2_{L^2(\bbS^{n-1})},
\]
for a constant $C$ depending on $n$ and $G$.
Combining this with \cref{fuglede:eq2,fuglede:eq3,fuglede:eq5}, we deduce
\begin{equation}\label{fuglede:eq7}
\vphi_\eps(t)
\leq (1+Cq_\eta(\eps))\norm{\grad_\tau\,u}_{L^2(\bbS^{n-1})}^2
+Ct\left(\norm{u}_{L^2(\bbS^{n-1})}^2+ \norm{\grad_\tau\,u}_{L^2(\bbS^{n-1})}^2\right),
\end{equation}
for some $C=C(n,G)$, which concludes this step.\\[4pt]
\textit{Step 3. Estimation of $\psi_\eps(t)$.}
By Fubini's theorem, we have
\begin{equation}\label{fuglede:eq8}
\psi_\eps(t)=2\int_{\partial B}\int_0^{1+tu(x)}\int_{1+tu(x)}^{+\infty}
r^{n-1}\rho^{n-1}\left(\int_{\partial B}G_\eps(rx-\rho y) \dH^{n-1}_y\right)
\dr\dd\rho\dH^{n-1}_x.
\end{equation}
Notice that for any rotation $R\in SO(n)$, $\abs{rx-\rho(R^{-1}y)}=\abs{r(Rx)-y}$, so that,
for any $\sigma\in\partial B$, making the change of variable $y'=R^{-1}y$ with $R$ a rotation
mapping $x$ to $\sigma$, by symmetry of $G$, we find
\[
\int_{\partial B} G_\eps(rx-\rho y)\dH_y^{n-1}
=\int_{\partial B} G(r\sigma-\rho y)\dH_y^{n-1}.
\]
Hence, the above integral does not depend on $x\in\partial B$, and by averaging over $\partial B$,
we obtain
\[
\int_{\partial B} G_\eps(rx-\rho y) \dH^{n-1}_y
=\frac{1}{P(B)}\iint_{\partial B\times\partial B} G_\eps(r\sigma-\rho y)\dH^{n-1}_\sigma\dH^{n-1}_y.
\]
Inserting this into \cref{fuglede:eq8} and using Fubini's theorem once again yields
\[
\begin{aligned}
\psi_\eps(t)
&=\frac{1}{P(B)}\int\limits_{\partial B}
\left(2\iint\limits_{\partial B\times\partial B}\int_0^{1+tu(x)}\int_{1+tu(x)}^{+\infty}
r^{n-1}\rho^{n-1}G_\eps(r\sigma-\rho y) \dr\dd\rho\dH^{n-1}_\sigma\dH^{n-1}_y\right)
\dH_x^{n-1}\\
&=\frac{1}{P(B)}\int_{\partial B}
\left(2\iint_{B_{1+tu(x)}\times B_{1+tu(x)}^\compl}G_\eps(y-z) \dy\dz\right)\dH_x^{n-1}\\
&=\frac{1}{P(B)}\int_{\partial B} \Per_{G_\eps}(B_{1+tu(x)})\dH_x^{n-1}.
\end{aligned}
\]
Then, by \cref{prp:boundPerG}, we deduce
\begin{equation}\label{fuglede:eq9}
\psi_\eps(t)
\leq \frac{1}{P(B)}\int_{\partial B} P(B_{1+tu(x)})
= \int_{\partial B} \big(1+tu(x)\big)^{n-1}\dH_x^{n-1}.
\end{equation}
Since $t<t_0<1/8$ and $\norm{u}_\infty\leq 1$, we have
\[
(1+tu(x))^{n-1}\leq 1+(n-1)tu(x)+(n-1)(n-2)\frac{t^2}{2}u(x)^2+Ct\abs{u(x)}^3,
\]
thus, by \cref{var1min:res1} of \cref{lem:var1min}, from \cref{fuglede:eq9} it follows
\begin{equation}\label{fuglede:eq10}
\begin{aligned}
\psi_\eps(t)
&\leq\int_{\partial B} (1+tu(x))^{n-1}\dH_x^{n-1}\\
&\leq P(B)+\frac{t^2}{2}\left((n-1)(n-2)-(n-1)^2\right)\int_{\partial B}
u^2\dH^{n-1}+t^3\int_{\partial B} \abs{u}^3\dH^{n-1}\\
&= P(B)-(n-1)\frac{t^2}{2}\norm{u}_{L^2(\bbS^{n-1})}^2+t^3\int_{\partial B} \abs{u}^3\dH^{n-1}.
\end{aligned}
\end{equation}
This concludes Step 3.\\[4pt]
\textit{Conclusion.}
Eventually, gathering \cref{fuglede:eq1}, \cref{fuglede:eq7} and \cref{fuglede:eq10}, we
find
\[
\begin{aligned}
\Per_{G_\eps}(E_t)
&\leq \Per_{G_\eps}(B)+\frac{t^2}{2}\left(\big(1+Cq_\eta(\eps)\big)
\norm{\grad_\tau\,u}_{L^2(\bbS^{n-1})}^2-(n-1)\norm{u}_{L^2(\bbS^{n-1})}^2\right)\\
&\hphantom{\leq \Per_{G_\eps}(B)+\frac{t^2}{2}\left(\norm{\grad_\tau\,u}\right.}
+Ct^3\left(\norm{\grad_\tau\,u}_{L^2(\bbS^{n-1})}^2+\norm{u}_{L^2(\bbS^{n-1})}^2\right),
\end{aligned}
\]
for some $C=C(n,G)$, provided $t_0=t_0(n)$ is chosen small enough, which concludes the proof.
\end{proof}

\subsection{Minimality of the unit ball}

In order to take advantage of \cref{lem:fuglede}, one should have that for a centered $t$-nearly
spherical set $E_t$ such that $\partial E_t=\left\{(1+tu(x))x~:~x\in\bbS^{n-1}\right\}$, the
quantity
\[
\big(1+Cq_\eta(\eps)\big)\norm{\grad_\tau\,u}_{L^2(\bbS^{n-1})}^2-(n-1)\norm{u}_{L^2(\bbS^{n-1})}^2
\]
from \cref{eq:resfuglede} controls $\norm{u}_{H^1(\bbS^{n-1})}^2$ for small $t$ and $\eps$. This is
the purpose of \cref{lem:var1min} in appendix.

With \cref{lem:fuglede} and Fuglede's result for the local perimeter, we deduce a lower
bound for $\calF_{\gamma,G_\eps}(E_t)-\calF_{\gamma,G_\eps}(B_1)$.

\begin{prp}\label{prp:fuglede2}
Assume that $G$ satisfies \crefrange{Hrad}{Hhighermoments}.
Then there exist positive constants $t_*$ and $\eps_3$ depending only on $n$, $G$ and $\gamma$ such
that the following holds. If $E_t$ is a centered $t$-nearly spherical set such that $\partial
E_t=\left\{(1+tu(x))x~:~x\in\bbS^{n-1}\right\}$ with $0<t<t_*$, then for every $0<\eps<\eps_3$, we
have
\[
\calF_{\gamma,G_\eps}(E_t)-\calF_{\gamma,G_\eps}(B_1)
\geq \frac{t^2}{16}(1-\gamma)
\left(\norm{\grad_\tau\,u}_{L^2(\bbS^{n-1})}^2+\norm{u}_{L^2(\bbS^{n-1})}^2\right).
\]
\end{prp}

\begin{proof}
Assume $0<\eps<\eps_3$, and $0<t<t_*$, where $t_*=t_*(n)$ and~$\eps_3=\eps_3(n,G,\gamma)$ will be
fixed later.
If $E_t$ is a centered $t$-nearly spherical set and $t_*<t_1$ as well, where $t_1=t_1(n)$ is given
by \cref{lem:var1min}, we have in particular
\begin{equation}\label{fugcor:eq0}
\norm{\grad_\tau\, u}_{L^2(\bbS^{n-1})}^2+\norm{u}_{L^2(\bbS^{n-1})}^2
\leq 4\left(\norm{\grad_\tau\,u}_{L^2(\bbS^{n-1})}^2-(n-1)\norm{u}_{L^2(\bbS^{n-1})}^2\right).
\end{equation}
\edit{Then, choosing $t_*<\ol{t}$ where $\ol{t}=\ol{t}(n)$ is given by \cref{lem:fugledeper}, we
have}
\[
P(E_t)\geq
P(B_1)+\frac{t^2}{2}\left(\norm{\grad_\tau\,u}_{L^2(\bbS^{n-1})}^2
-(n-1)\norm{u}_{L^2(\bbS^{n-1})}^2\right) -Ct^3\left(\norm{\grad_\tau\,
u}_{L^2(\bbS^{n-1})}^2+\norm{u}_{L^2(\bbS^{n-1})}^2\right).
\]
for some positive constant $C=C(n)$.
Then, assuming that $t_*<t_0(n)$ as well, and $\eps<\ol{\eps}_5(n,G)$, where $t_0$ and
$\ol{\eps}_5$ are given by \cref{lem:fuglede}, we find
\begin{equation}\label{fugcor:eq1}
\begin{aligned}
&\calF_{\gamma,G_\eps}(E_t)-\calF_{\gamma,G_\eps}(B_1)\\
&\hphantom{\calF}
\geq \frac{t^2}{2}
\Big[ (1-\gamma)\Big(\norm{\grad_\tau\,u}_{L^2(\bbS^{n-1})}^2-(n-1)\norm{u}_{L^2(\bbS^{n-1})}^2\Big)
-C'q_\eta(\eps)\norm{\grad_\tau\,u}_{L^2(\bbS^{n-1})}^2\Big]\\
&\hphantom{\calF}\hphantom{\geq \frac{t^2}{2} \Big[ (1-\gamma)}
-C't^3\Big(\norm{\grad_\tau\,u}_{L^2(\bbS^{n-1})}^2+\norm{u}_{L^2(\bbS^{n-1})}^2\Big),
\end{aligned}
\end{equation}
for some positive constant $C'=C'(n,G)$, with $q_\eta(\eps)\to 0$ as $\eps\to 0$.
From \cref{fugcor:eq0,fugcor:eq1}, it follows
\[
\calF_{\gamma,G_\eps}(E_t)-\calF_{\gamma,G_\eps}(B_1)
\geq \frac{1}{4}\left(\frac{t^2}{2}\big((1-\gamma)-C'q_\eta(\eps)\big)-(C+C')t^3\right)
\left(\norm{\grad_\tau\,u}_{L^2(\bbS^{n-1})}^2+\norm{u}_{L^2(\bbS^{n-1})}^2\right).
\]
Eventually, choosing $\eps_3$, $t_*$ small enough depending only on $n$, $G$ and $\gamma$, for any
$0<\eps<\eps_3$ and~$0< t< t_*$, we have
\[
\calF_{\gamma,G_\eps}(E_t)-\calF_{\gamma,G_\eps}(B_1) \geq \frac{t^2}{16}(1-\gamma)
\left(\norm{\grad_\tau\,u}_{L^2(\bbS^{n-1})}^2+\norm{u}_{L^2(\bbS^{n-1})}^2\right),
\]
which proves the result.
\end{proof}

An immediate consequence of \cref{prp:fuglede2} is that the unit ball is the only minimizer, up to
translations, of \cref{minrp}, among $t$-nearly spherical sets whenever $t<t_*(n,G,\gamma)$ and
$\eps<\eps_3(n,G,\gamma)$, that is, \cref{mainthm:minballsph}. In dimension $n=2$, by
\cref{mainthm:minnearlycirc,mainthm:minballsph}, choosing $\eps_A<\min(\eps_2,\eps_3)$
such that $t(\eps)<t_*$ for every $0<\eps<\eps_A$, we readily obtain~\cref{mainthm:mindisk}.

\begin{rmk}\label{rmk:depepsstar}
\edit{Tracking the dependence in $\gamma$, we see that $\eps_3$ is chosen such that
$q_\eta(\eps)\leq C(n,G)(1-\gamma)$ for all $0<\eps<\eps_3$, and $t_*= C(n,G)(1-\gamma)$. As for the
quantity $P(B_1)-\Per_{\eps}(B_1)$, we have
\[
q_\eta(\eps) \sim C(n,G)\eps^2
\]
as $\eps$ vanishes, thus $\eps_3=C(n,G)(1-\gamma)^{1/2}$. Then since $t(\eps)$ of
\cref{mainthm:minnearlycirc} is bounded from above by $C(n)\delta(\eps)^{1/2}$ (see
\cref{prp:minsph}) and $\eps_A$ of \cref{mainthm:mindisk} is chosen such that $t(\eps)<t_*$ for all
$0<\eps<\eps_A$, we get the constraint
\[
\eps_A\leq C(n,G) \frac{(1-\gamma)^{\frac{5}{2}}}{\gamma^{\frac{1}{2}}},
\]
where we used \cref{rmk:distball} to estimate $\delta(\eps)$.
Since $\eps_A$ is also chosen smaller than $\min(\eps_2,\eps_3)$, if
$G(x)=O(\abs{x}^{-(n+1+\beta)})$ at infinity, by \cref{rmk:depeps12}, we find for $\gamma\geq 1/2$,
that we can choose}
\[
\edit{\eps_A=C(n,G)(1-\gamma)^{\max\left(\frac{5}{2},\frac{1}{2} +\frac{1}{\beta}\right)}.}
\]
\end{rmk}

\appendix

\section{Additional computations for \texorpdfstring{{\cref{sec:fuglede}}}{Section 4}}

In the following lemma, we establish some general inequalities on functions $u:\bbS^{n-1}\to \R$
describing centered nearly-spherical sets. For this, we need to recall a few basic facts and
notation on spherical harmonics. For $k\geq 0$, we denote by $\calS_k$ the subspace of spherical
harmonics of degree $k$ (i.e., restrictions to $\bbS^{n-1}$ of polynomials of degree $k$ in $\R^n$),
which is a finite-dimensional vector space of degree $d(k)$. 
Let $(Y_k^i)_{i\in\{1,\dotsc,d(k)\}}$ be an orthonormal basis of $\calS_k$ for the standard scalar
product of $L^2(\bbS^{n-1})$. When there can be no confusion, we write $Y_k$ for a generic vector of
the basis of $\calS_k$. It is well-known that the family $(Y_k^i)_{k\in\N}^{i\in \{1,\dotsc,d(k)\}}$
is a Hilbert basis of $L^2(\bbS^{n-1})$ which diagonalizes the Laplace-Beltrami operator on the
sphere, and the eigenvalue associated with $Y_k^i$ is $l_k\coloneqq l_k=k(k+n-2)$, for
all~$i\in\{1,\dotsc,d(k)\}$. We recall that $d(0)=1$, $d(1)=n$, and that the $Y_1^i$ may be chosen
colinear to $x\mapsto x_i$, for example~$Y_1^i=\abs{B_1}^{-1/2} x_i$ for all~$i\in\{1,\dotsc,n\}$.

\begin{lem}\label{lem:var1min}
There exist positive constants $t_1=t_1(n)$ and $C=C(n)$ such that the following holds.
If $E_t$ is a centered $t$-nearly spherical set such that $\partial
E_t=\left\{(1+tu(x))x~:~x\in\bbS^{n-1}\right\}$ with $0<t<t_1$, then we have
\begin{equation}\label{var1min:res1}
\abs*{t\int_{\bbS^{n-1}} u\dH^{n-1}+(n-1)\frac{t^2}{2}\int_{\bbS^{n-1}} u^2 \dH^{n-1}} \leq
Ct^3\int_{\bbS^{n-1}} \abs{u}^3 \dH^{n-1},
\end{equation}
and
\begin{equation}\label{var1min:res2}
\begin{aligned}
\frac{1}{2}\norm{\grad_\tau\, u}_{L^2(\bbS^{n-1})}^2-(n-1)\norm{u}_{L^2(\bbS^{n-1})}^2
\geq \frac{1}{2}\norm{u}_{L^2(\bbS^{n-1})}^2.
\end{aligned}
\end{equation}
\end{lem}

\begin{proof}
Since $\abs{E_t}=\abs{B}$, we have
\[
n\abs{E_t}=\int_{\partial B} (1+tu(x))^n\dH^{n-1}=\int_{\partial B} 1\dH^{n-1}=n\abs{B}.
\]
Thus, writing
\[
(1+tu(x))^n-\left(1+ntu(x)+\frac{(n-1)}{2}t^2u(x)^2\right)=\sum_{k=3}^n \binom{n}{k} t^k u(x)^k,
\]
we deduce \cref{var1min:res1}, choosing e.g. $t_1=1$ and $C$ depending only on $n$.
There remains to show \cref{var1min:res2}.
We decompose $u$ in spherical harmonics
\begin{equation}\label{fuglede:eq20}
u=\sum_{k=0}^{+\infty} \sum_{i=1}^{d(k)} a_k^i(u)\,Y_k^i,
\end{equation}
so that
\[
\norm{u}^2_{L^2(\bbS^{n-1})}=\sum_{k=0}^{+\infty} \sum_{i=1}^{d(k)} a_k^i(u)^2, \qquad
\norm{\grad_\tau\, u}^2_{L^2(\bbS^{n-1})}=\sum_{k=1}^{+\infty}\sum_{i=1}^{d(k)} l_k\, a_k^i(u)^2.
\]
Recall that $d(0)=1$, $Y_0^1$ is constant, $d(1)=n$, and that $Y_1^i$ is colinear to
$x\mapsto x_i$.
Since $l_1=n-1$ and~$l_k\geq 2n$ for $k\geq 2$, it follows
\begin{equation}\label{fuglede:eq22}
\begin{aligned}
&\frac{1}{2}\norm{\grad_\tau\,u}_{L^2(\bbS^{n-1})}^2-(n-1)\norm{u}_{L^2(\bbS^{n-1})}^2\\
&\hphantom{\norm{u}_{L^2(\bbS^{n-1})}}
\geq \norm{u}_{L^2(\bbS^{n-1})}^2-\frac{(n+1)}{2}\left(\sum_{i=1}^{n} a_1^i(u)^2\right)
-(n-1)a_0^1(u)^2.
\end{aligned}
\end{equation}
On one hand,
\[
a_0^1(u)=\frac{1}{\abs{\bbS^{n-1}}} \int_{\bbS^{n-1}} u\dH^{n-1},
\]
so that by \cref{var1min:res1}, we have
\begin{equation}\label{fuglede:eq23}
\abs{a_0^1(u)}\leq \frac{2t}{\abs{\bbS^{n-1}}} \norm{u}_{L^2(\bbS^{n-1})}^2
\leq\frac{1}{2\sqrt{n-1}}\norm{u}_{L^2(\bbS^{n-1})},
\end{equation}
up to choosing $t_1=t_1(n)$ small enough, and using $t<t_1$ and
$\norm{u}_{L^\infty(\bbS^{n-1})}\leq 1$.
On the other hand, the barycenter condition
\[
\int_{E_t} x\dx=0
\]
gives
\[
\int_{\bbS^{n-1}} x_i(1+tu(x))^{n}\dH_x^{n-1}=0,\qquad\forall i\in\{1,\dotsc,n\}.
\]
Using the binomial formula again and $Y_1^i(x)=\abs{B_1}^{-1/2}x_i$, we obtain, for
$i\in\{1,\dotsc,n\}$,
\[
a_1^i(u)=\frac{1}{\sqrt{\abs{B_1}}}\int_{\bbS^{n-1}} x_i u(x)\dH_x^{n-1}
=-\frac{1}{n\sqrt{\abs{B_1}}}\sum_{k=2}^n \binom{n}{k} t^k\int_{\bbS^{n-1}} x_i u(x)^k\dH_x^{n-1}.
\]
Next, using $\norm{u}_{L^\infty(\bbS^{n-1})}\leq 1$, Cauchy-Schwarz inequality, and
$n\abs{B_1}=\abs{\bbS^{n-1}}$, we get
\[
\abs{a_i(u)}\leq \frac{2^n}{n} t^2\norm{u}_{L^2(\bbS^{n-1})}.
\]
Whence, choosing $t_1$ even smaller, but still depending only on $n$, we may assume
\begin{equation}\label{fuglede:eq24}
\sum_{i=1}^{n}~ a_1^i(u)^2 \leq \frac{1}{2(n+1)}\norm{u}_{L^2(\bbS^{n-1})}^2.
\end{equation}
Gathering \cref{fuglede:eq22,fuglede:eq23,fuglede:eq24}, we find
\[
\frac{1}{2}\norm{\grad_\tau\, u}_{L^2(\bbS^{n-1})}^2-(n-1)\norm{u}_{L^2(\bbS^{n-1})}^2
\geq \frac{1}{2}\norm{u}_{L^2(\bbS^{n-1})}^2,
\]
which proves \cref{var1min:res2} and concludes the proof.
\end{proof}

We establish a technical lemma to control terms of the form
\[
\iint_{\bbS^{n-1}\times\bbS^{n-1}} \int_0^1\int_0^1 (u(x)-u(y))^2
k_\eps(\abs{\Phi_{tu}(x,y,r,\rho)})\dr\dd\rho\dH^{n-1}_x\dH^{n-1}_y,
\]
by the $H^1(\bbS^{n-1})$ norm of $u$, where $\Phi_{tu}(x,y,r,\rho)$ is a small perturbation of
$(x-y)$, and $k_\eps$ are suitable rescalings of a nonnegative kernel.

\begin{lem}\label{lem:techlem1}
Let $k:(0,+\infty)\to[0,+\infty)$ be a measurable function such that $k(r)\leq C_0 r^{-(n+1)}$
on $(R_0,+\infty)$ for some positive constants $C_0$, $R_0$, and
\[
I_k^1 \coloneqq \int_{\Ren} \abs{x}k(\abs{x})\dx <+\infty.
\]
Let us define the rescaling $k_\eps(r)\coloneqq \eps^{-(n+1)}k(\eps^{-1}r)$, $r,~\eps>0$. 
For $v\in\Lip(\bbS^{n-1})$, we define a map~$\Phi_v:\left(\bbS^{n-1}\right)^2\times (0,1)^2$ by
\begin{equation}\label{eq:defPhi}
\Phi_v(x,y,r,\rho)=(x-y)+\Big[\big(rv(x)+(1-r)v(y))\big)x-\big(\rho v(y)+(1-\rho)v(x)\big)y\Big].
\end{equation}
Then there exist positive constants $\ol{t}_1=\ol{t}_1(n)$,
$\ol{C}_0=\ol{C}_0(n,C_0,I_k^1)$ and $\ol{\eps}_6=\ol{\eps}_6(n,R_0)$ such that the
following holds.
For any $u\in\Lip(\bbS^{n-1})$ such that $\norm{u}_{L^\infty(\bbS^{n-1})}\leq 1$,
$\norm{\grad_\tau\, u}_{L^\infty(\bbS^{n-1})}\leq 1$, and any $t\in(0,\ol{t}_1)$, we have
\begin{multline}\label{techlem1:skernel}
\iint_{\bbS^{n-1}\times\bbS^{n-1}} \int_0^1\int_0^1 (u(x)-u(y))^2
k_\eps(\abs{\Phi_{tu}(x,y,r,\rho)})\dr\dd\rho\dH^{n-1}_x\dH^{n-1}_y\\
\leq \ol{C}_0\left(\int_{\bbS^{n-1}} \abs{\grad_\tau\, u}^2\dH^{n-1}+\int_{\bbS^{n-1}}
\abs{u}^2\dH^{n-1}\right).
\end{multline}
\end{lem}

\begin{proof}
Let $\ol{t}_1>0$ to be chosen later. In the proof, we assume that $0<t<\ol{t}_1$.
We work in local coordinates, proceeding in two steps.\\[4pt]
\textit{Step 1.}
Let us denote by $D_r(x)$ the open ball of radius $r$ centered at $x$ in $\R^{n-1}$. We first show
that for any $\wt{u}\in\Lip(D_2)$ such that $\norm{\wt{u}}_{L^\infty(D_2)}+\norm{\grad
\wt{u}}_{L^\infty(D_2)}\leq M$, we have
\begin{equation}\label{techlem1:eqstep1}
\int_{D_2}\int_{D_2}\int_0^1\int_0^1 (\wt{u}(x)-\wt{u}(y))^2
k_\eps(\abs{\Phi_{t\wt{u}}(x,y,r,\rho)})\dr\dd\rho\dx\dy
\leq C\int_{D_2} \abs{\grad \wt{u}}^2\dx,
\end{equation}
for all $\eps>0$, provided $t<\ol{t}_1$, and $\ol{t}_1$ is chosen small enough depending only on
$n$ and $M$, and $C=C(n,I_k^1)$. Here, by a slight abuse of notation, $\Phi_v$ denotes the map from
$\left(\R^{n-1}\right)^2\times(0,1)^2$ whose expression is given by \cref{eq:defPhi}.
For $v\in\Lip(D_2)$, let us define the map $E_v:\left(\R^{n-1}\right)^2\times (0,1)^2$ by
\[
E_v(x,y,r,\rho) \coloneqq \big(rv(x)+(1-r)v(y))\big)x-\big(\rho v(y)+(1-\rho)v(x)\big)y,
\]
so that $\Phi_v(x,y,r,\rho)=x-y+E_v(x,y,r,\rho)$.
Due to the $L^\infty$ bound on $\wt{u}$ and $\grad \wt{u}$, we easily see that the
maps~$E_{t\wt{u}}$ converge uniformly to $0$ on $(D_2)^2\times (0,1)^2$ as $t$ vanishes.
In fact, we have
\begin{equation}\label{techlem1:eq1}
\arraycolsep=1.4pt
\begin{array}{rl}
\abs*{E_{t\wt{u}}(x,y,r,\rho)}
&\leq t\left(2\norm{\wt{u}}_{L^\infty(D_2)}+\norm{\grad\wt{u}}_{L^\infty(D_2)}\right)
\abs{x-y}\\[4pt]
&\leq 2M\ol{t}_1\abs{x-y}
\end{array}
\qquad\forall x,y\in D_2,\,\forall r,\rho\in(0,1).
\end{equation}
In particular, choosing $\ol{t}_1<1/(4M)$,
\begin{equation}\label{techlem1:eq2}
\frac{\abs{x-y}}{2}\leq \abs{\Phi_{t\wt{u}}(x,y,r,\rho)} \leq 2\abs{x-y},\quad\forall x,y\in
D_2,\,\forall r,\rho\in(0,1).
\end{equation}
Integrating on lines, using Cauchy-Schwarz inequality and Fubini's theorem, we have
\begin{equation}\label{techlem1:eq3}
\begin{aligned}
&\int_{D_2}\int_{D_2}\int_0^1\int_0^1 (\wt{u}(x)-\wt{u}(y))^2
k_\eps(\abs{\Phi_{t\wt{u}}(x,y,r,\rho)})\dr\dd\rho\dx\dy\\
&\hphantom{\int_{D_2}\int_{D_2}}
\leq \int_0^1 \int_0^1\int_0^1 \int_{D_2}\int_{D_2} \abs{\grad\wt{u}(x+s(y-x))}^2 \abs{x-y}^2
k_\eps(\abs{\Phi_{t\wt{u}}(x,y,r,\rho)})\dx\dy\dd s\dr\dd\rho.
\end{aligned}
\end{equation}
Now let us focus on the integral on $D_2\times D_2$, fixing $s\in(0,1/2)$ and
$r,\rho\in(0,1)$. By \cref{techlem1:eq2}, our choice of~$\ol{t}_1$, and the fact that
$0<t<\ol{t}_1$, it follows
\begin{equation}\label{techlem1:eq4}
\begin{aligned}
&\int_{D_2}\int_{D_2} \abs{\grad\wt{u}(x+s(y-x))}^2 \abs{x-y}^2
k_\eps(\abs{\Phi_{t\wt{u}}(x,y,r,\rho)})\dx\dy\\
&\hphantom{\int_{D_2}\int_{D_2}}
\leq 4 \int_{D_2}\int_{D_2} \abs{\grad\wt{u}(x+s(y-x))}^2 \abs{\Phi_{t\wt{u}}(x,y,r,\rho)}^2
k_\eps(\abs{\Phi_{t\wt{u}}(x,y,r,\rho)})\dx\dy.
\end{aligned}
\end{equation}
We wish to make the change of variables
\[
\Psi_{t\wt{u},r,\rho,s}(x,y)=(\Phi_{t\wt{u}}(x,y,r,\rho),(1-s)x+sy)\eqqcolon (x',y').
\]
First, let us remark that $\Psi_{t\wt{u},r,\rho,s}$ is an injection whenever
$\ol{t}_1<1/(4M)$. Indeed,
\begin{equation}\label{techlem1:eq5}
\begin{aligned}
\Psi_{t\wt{u},r,\rho,s}(x_1,y_1)=
\Psi_{t\wt{u},r,\rho,s}(x_2,y_2)
&\iff
\begin{cases}
x_1-x_2=y_1-y_2+E_{t\wt{u}}(x_2,y_2)-E_{t\wt{u}}(x_1,y_1)\\
(1-s)(x_1-x_2)=s(y_2-y_1)
\end{cases}\\
&\iff
\begin{cases}
x_1-x_2=s\left[E_{t\wt{u}}(x_2,y_2)-E_{t\wt{u}}(x_1,y_1)\right]\\
(1-s)(x_1-x_2)=s(y_2-y_1).
\end{cases}
\end{aligned}
\end{equation}
We compute
\[
\begin{aligned}
&E_{t\wt{u}}(x_2,y_2)-E_{t\wt{u}}(x_1,y_1)\\
&\hphantom{E_{t\wt{u}}}
=t\Big[\big(r\wt{u}(x_2)+(1-r)\wt{u}(y_2)\big)x_2-\big(r\wt{u}(x_1)+(1-r)\wt{u}(y_2)\big)x_1\Big]\\
&\hphantom{E_{t\wt{u}}}\hphantom{t\Big[\big(r\wt{u}}
-t\Big[\big(\rho \wt{u}(y_2)+(1-\rho)\wt{u}(x_2)\big)y_2-\big(\rho\wt{u}(y_1)
+(1-\rho)\wt{u}(x_1)\big)y_1\Big]\\
&\hphantom{E_{t\wt{u}}}
=t\Big[\big(r\wt{u}(x_2)+(1-r)\wt{u}(y_2)\big)(x_2-x_1)
+\big(r(\wt{u}(x_2)-\wt{u}(x_1))+(1-r)(\wt{u}(y_1)-\wt{u}(y_2))\big)x_1\Big]\\
&\hphantom{E_{t\wt{u}}}\hphantom{t\Big[\big(r\wt{u}}
-t\Big[\big(\rho \wt{u}(y_2)+(1-\rho)\wt{u}(x_2)\big)(y_2-y_1)
+\big(\rho(\wt{u}(y_2)-\wt{u}(y_1))+(1-\rho)(\wt{u}(x_2)-\wt{u}(x_1))\big)y_1\Big],
\end{aligned}
\]
thus, using the inequality $\norm{\wt{u}}_{L^\infty(D_2)}+\norm{\grad\wt{u}}_{L^\infty(D_2)}\leq M$
and the fact that $x_1,y_1\in D_2$, we find
\begin{multline}\label{techlem1:eq6}
\abs*{E_{t\wt{u}}(x_2,y_2)-E_{t\wt{u}}(x_1,y_1)}
\leq t\left(\norm{\wt{u}}_{L^\infty(D_2)}+2\norm{\grad\wt{u}}_{L^\infty(D_2)}\right)
\left(\abs{x_2-x_1}+\abs{y_2-y_1}\right)
\\\leq 2M\ol{t}_1\left(\abs{x_2-x_1}+\abs{y_2-y_1}\right).
\end{multline}
If $(1-s)(x_1-x_2)=s(y_2-y_1)$, \cref{techlem1:eq6} implies
\[
s\abs*{E_{t\wt{u}}(x_2,y_2)-E_{t\wt{u}}(x_1,y_1)}\leq 2M\ol{t}_1\abs{x_2-x_1},
\]
and plugging this into \cref{techlem1:eq5}, we see that $\Psi_{t\wt{u},r,\rho,s}$ is
injective if $\ol{t}_1<1/(4M)$. Then, a simple computation gives
\[
D\Psi_{t\wt{u},r,\rho,s}(x,y)
=\begin{pmatrix}
A(x,y) & -B(x,y)\\
(1-s) I_n & s I_n
\end{pmatrix},
\]
where $I_n$ denotes the identity matrix in $\Ren$, and for almost every $x,y\in D_2$, 
\[
A(x,y)\coloneqq I_n+t\Big[\big(r\wt{u}(x)+(1-r)\wt{u}(y)\big)I_n+\big(rx-(1-\rho)y\big)\otimes
\grad \wt{u}(x)\Big]
\]
and
\[
B(x,y)\coloneqq I_n+t\Big[\big(\rho \wt{u}(y)+(1-\rho)\wt{u}(x)\big)I_n+\big(\rho
y-(1-r)x\big)\otimes \grad \wt{u}(y)\Big].
\]
Since $\det D\Psi_{t\wt{u},r,\rho,s}=\det(sA+(1-s)B)$ and
\[
sA(x,y)+(1-s)B(x,y)
=I_n+t\Big[I_n+s\big(rx-(1-\rho)y\big)\otimes\grad \wt{u}(x)+(1-s)\big((1-r)x-\rho
y\big)\otimes\grad\wt{u}(y)\Big],
\]
we see that, choosing $\ol{t}_1$ even smaller if needed, depending only on $n$ and $M$, we have
\[
\abs{\det D\Psi_{t\wt{u},r,\rho,s}(x,y)}\geq \frac{1}{2}.
\]
Hence, making the change of variables $(x',y')=\Psi_{t\wt{u},r,\rho,s}(x,y)$ in \cref{techlem1:eq4}
yields
\begin{equation}\label{techlem1:eq7}
\begin{aligned}
&\int_{D_2}\int_{D_2} \abs{\grad\wt{u}(x+s(y-x))}^2 \abs{x-y}^2
k_\eps(\abs{\Phi_{t\wt{u}}(x,y,r,\rho)})\dx\dy\\
&\hphantom{\int_{D_2}\int_{D_2}}
\leq 8 \iint_{\Psi_{t\wt{u},r,\rho,s}(D_2\times D_2)} \abs{\grad\wt{u}(y')}^2 \abs{x'}^2
k_\eps(\abs{x'})\dx'\dy'\\
&\hphantom{\int_{D_2}\int_{D_2}}
\leq 8 \left(\int_{D_2} \abs{\grad\wt{u}(y)}^2\dy\right)\left(\int_0^8 t^{n}
k_\eps(t)\dt\right),
\end{aligned}
\end{equation}
where we used the fact that $\Psi_{t\wt{u},r,\rho,s}(D_2\times D_2)\subsq D_8\times D_2$ for
the last inequality, in view of \cref{techlem1:eq2}.
Plugging \cref{techlem1:eq7} into \cref{techlem1:eq3} and making the change of variable $t'=t/\eps$
gives \cref{techlem1:eqstep1}, by the assumptions on $k$, which concludes this step.\\[4pt]
\textit{Step 2.}
We split the domain of integration in the left-hand side of \cref{techlem1:skernel} in
$\left\{x,y\in\bbS^{n-1}~:~ \abs{x-y}>1/4\right\}$ and
$\left\{x,y\in\bbS^{n-1}~:~\abs{x-y}\leq 1/4\right\}$. We first treat the contribution of
distant pairs $(x,y)$.
Choosing $\ol{t}_1$ small enough depending only on $n$, we have
\[
\abs{\Phi_{tu}(x,y,r,\rho)}\geq \frac{\abs{x-y}}{2},\qquad\forall x,y\in\bbS^{n-1},\,\forall
r,\rho\in(0,1),
\]
so that, by the squared triangle inequality and by symmetry,
\begin{equation}\label{techlem2:eq5b}
\begin{aligned}
&\iint\limits_{\substack{\bbS^{n-1}\times\bbS^{n-1}\\\{\abs{x-y}>\frac{1}{4}\}}}
\int_0^1\int_0^1 (u(x)-u(y))^2
k_\eps(\abs{\Phi_{tu}(x,y,r,\rho)})\dr\dd\rho\dH^{n-1}_x\dH^{n-1}_y\\
&\hphantom{\iint\limits_{\substack{\bbS^{n-1}}}}
\leq 2\int_{\bbS^{n-1}} \abs{u(x)}^2
\Bigg(\int\limits_{\substack{\bbS^{n-1}\\\{\abs{x-y}>\frac{1}{4}\}}}
\int_0^1\int_0^1 k_\eps(\abs{\Phi_{tu}(x,y,r,\rho)})\dr\dd\rho\dH^{n-1}_y
\Bigg) \dH_x^{n-1}\\
&\hphantom{\iint\limits_{\substack{\bbS^{n-1}}}}
\leq 2\abs{\bbS^{n-1}} \left(\int_{\bbS^{n-1}} \abs{u}^2\dH^{n-1}\right)\left(\sup_{r>\frac{1}{8}}~
k_\eps(r)\right).
\end{aligned}
\end{equation}
Then, choosing $\ol{\eps}_6=1/(8R_0)$, and using that
$k(r)\leq C_0r^{-(n+1)}$ for all $r\in(R_0,+\infty)$, for any $0<\eps<\ol{\eps}_6$,
\cref{techlem2:eq5b} implies
\begin{equation}\label{techlem2:eq6}
\iint\limits_{\substack{\bbS^{n-1}\times\bbS^{n-1}\\\{\abs{x-y}>\frac{1}{4}\}}}
\int_0^1\int_0^1 (u(x)-u(y))^2
k_\eps(\abs{\Phi_{tu}(x,y,r,\rho)})\dr\dd\rho\dH^{n-1}_x\dH^{n-1}_y
\leq C\int_{\bbS^{n-1}} \abs{u}^2\dH^{n-1},
\end{equation}
for some $C=C(n,C_0)$.\\
There remains to estimate the integral over the domain
$\mathcal{M}\coloneqq\left\{x,y\in\bbS^{n-1}~:~\abs{x-y}<1/4\right\}$. For this, we cover
$\mathcal{M}$ by a finite number $N(n)$ of~$\mathcal{M}_i\coloneqq \bbS^{n-1}_+(x_i)\times
\bbS^{n-1}_+(x_i)$, where for each $i\in\{1,\dotsc,N\}$, $x_i\in\bbS^{n-1}$, and $\bbS^{n-1}_+(x_i)$
denotes the hemisphere with center $x_i$. Using the stereographic projection $\Pi_i$ with
respect to $-x_i$, we map $\bbS^{n-1}_+(x_i)$ to $D_2\subsq \R^{n-1}$. By the changes of
variables $\xi=\Pi_i(x)$, $\zeta=\Pi_i(y)$, setting $\wt{u}_i\coloneqq u\circ\Pi_i^{-1}$, we
have~$\norm{\wt{u}_i}_{L^\infty(D_2)}+\norm{\grad\wt{u}_i}_{L^\infty(D_2)}\leq C(n)$ since
$\norm{u}_{L^\infty(\bbS^{n-1})}\leq 1$ and $\norm{\grad_\tau\, u}_{L^\infty(\bbS^{n-1})}\leq 1$.
Applying Step~1 with $\wt{u}=\wt{u}_i$, for any $\eps>0$, we obtain
\[
\begin{aligned}
&\iint_{\mathcal{M}_i}\int_0^1\int_0^1 (u(x)-u(y))^2
k_\eps(\abs{\Phi_{tu}(x,y,r,\rho)})\dr\dd\rho\dH_x^{n-1}\dH_y^{n-1}\\
&\hphantom{\iint_{\mathcal{M}_i}}
\leq C\iint_{D_2\times D_2}\int_0^1\int_0^1 (\wt{u}_i(\xi)-\wt{u}_i(\zeta))^2
k_\eps(\abs{\Phi_{t\wt{u}_i}(\xi,\zeta,r,\rho)})\dr\dd\rho\dd\xi\dd\zeta\\
&\hphantom{\iint_{\mathcal{M}_i}}
\leq C\int_{D_2} \abs{\grad \wt{u}_i}^2\dx
\leq C\int_{\bbS^{n-1}_+(x_i)} \abs{\grad_\tau\, u}^2\dH^{n-1},
\end{aligned}
\]
whenever $t<\ol{t}_1$, where $\ol{t}_1=\ol{t}_1(n)$ and $C=C(n,I_k^1)$.
Summing these estimates over $i\in\{1,\dotsc,N(n)\}$ and using \cref{techlem2:eq6}, we conclude the
proof.
\end{proof}

\addtocontents{toc}{\protect\setcounter{tocdepth}{1}}

\subsection*{Acknowledgments}

B. Merlet and M. Pegon are partially supported by the Labex CEMPI (ANR-11-LABX-0007-01).

\printbibliography 

\vspace{0.5cm}


\end{document}

%% file: 2d_figure1.tex
\usetikzlibrary{math}
\usetikzlibrary{patterns}
\usetikzlibrary{arrows.meta}

\tikzmath{
\scale=2;
\textscale=0.8;
}

\tikzset{
    partial ellipse/.style args={#1:#2:#3}{
        insert path={+ (#1:#3) arc (#1:#2:#3)}
    }
}

\begin{tikzpicture}[>=Latex]

\draw[color=black!10!red, pattern=north west lines, pattern color=black!10!red] (1.075*\scale,0) -- (0.82*\scale,0.37*\scale) arc(125:235:0.45*\scale) -- cycle;
\draw[color=blue!90] (1.075*\scale,0) circle (0.45*\scale);

\draw[color=black,dash pattern=on 3pt off 1pt] (0,0) circle (0.9*\scale);
\draw[color=black,dash pattern=on 3pt off 1pt] (0,0) circle (1.2*\scale);

\draw[fill=black] (0,0) circle (0.02*\scale) node[scale=\textscale,anchor=south] {$O$};
\draw[fill=black] (1.075*\scale,0) circle (0.02*\scale) node[scale=\textscale,outer xsep=-1pt,anchor=south west] {$x$};

\draw[color=black!10!red,fill=black!10!red] (0.82*\scale,0.37*\scale) circle (0.02*\scale);
\draw[color=black!10!red,fill=black!10!red] (0.82*\scale,-0.37*\scale) circle (0.02*\scale);

\draw (-0.70*\scale,0.76*\scale) node[scale=\textscale] {$E$};
\draw[color=black] (-0.4*\scale,-0.6*\scale) node[scale=\textscale] {$B_{1-\delta(\varepsilon)}$};
\draw[color=black] (-0.95*\scale,-1*\scale) node[scale=\textscale] {$B_{1+\delta(\varepsilon)}$};
\draw[color=blue!90] (1.42*\scale,0.46*\scale) node[scale=\textscale] {$B_r(x)$};
\draw[color=black!10!red] (0.48*\scale,0.2*\scale) node[scale=\textscale,anchor=south] {$C(x,r)$};
\draw[color=black!10!red] (0.85*\scale,0.38*\scale) node[scale=\textscale,anchor=south] {$A$};
\draw[color=black!10!red] (0.83*\scale,-0.4*\scale) node[scale=\textscale,anchor=east] {$B$};


\draw[smooth,samples=100,domain=-180:180,variable=\x,thick] plot ({(cos(\x)*(1.07+0.15*sin(\x-45)^3)+0.055)*\scale}, {(sin(\x)*(1.05+0.15*sin(\x-45)^3)-0.04)*\scale});

\draw[color=gray,dash pattern=on 1pt off 1pt] (-0.1*\scale,0.60*\scale) rectangle (1.65*\scale,-0.55*\scale);

\draw[->,semithick] (1.75*\scale,0) -- (2.5*\scale,0);


\filldraw[color=black!40, fill=black!5] (4.2*\scale,0) [partial ellipse=180:123:{0.2*\scale} and {0.2*\scale}];
\fill[black!5] (4.2*\scale,0) -- (4.09*\scale,0.165*\scale) -- (4*\scale,0) -- cycle;

\draw[red] (3.79*\scale,0.1*\scale) -- (3.69*\scale,0.1*\scale) -- (3.69*\scale,0);

\draw[color=black!10!red] (4.2*\scale,0) -- (3.78*\scale,0.63*\scale);


\draw[color=black] (2.7*\scale,0) -- (4.2*\scale,0);

\draw[color=black,semithick] (3.79*\scale,0) -- (3.79*\scale,0.628*\scale);

\draw[color=black] (2.7*\scale,0) -- (3.79*\scale,0.628*\scale);

\draw[color=blue!90] (4.2*\scale,0) circle (0.753*\scale);

\draw[color=black,dash pattern=on 3pt off 1pt] (2.7*\scale,0) [partial ellipse=-65:65:{1.26*\scale} and {1.26*\scale}];

\draw[color=black!10!red,fill=black!10!red] (3.79*\scale,0.628*\scale) circle (0.02*\scale);
\draw[color=black!10!red,fill=black!10!red] (3.79*\scale,-0.628*\scale) circle (0.02*\scale);

\draw[<->] (3.77*\scale,-0.08*\scale) -- (4.2*\scale,-0.08*\scale);

\draw[fill=black] (2.7*\scale,0) circle (0.02*\scale) node[scale=\textscale,anchor=south] {$O$};
\draw[fill=black] (4.2*\scale,0) circle (0.02*\scale) node[scale=\textscale,outer xsep=-1pt,anchor=south west] {$x$};
\draw (3.2*\scale,0.5*\scale) node[scale=\textscale] {$1-\delta(\varepsilon)$};
\draw (3.8*\scale,0.4*\scale) node[scale=\textscale,anchor=east] {$h$};
\draw (4.095*\scale,0.16*\scale) node[scale=\textscale,anchor=east] {$\frac{\varphi}{2}$};
\draw (4*\scale,-0.08*\scale) node[scale=\textscale,anchor=north] {$e$};
\draw (4*\scale,0.3*\scale) node[scale=\textscale,anchor=south west] {$r$};
\draw[color=black!10!red] (3.81*\scale,0.65*\scale) node[scale=\textscale,anchor=south] {$A$};
\draw[color=black!10!red] (3.81*\scale,-0.65*\scale) node[scale=\textscale,anchor=north] {$B$};

\end{tikzpicture}